\documentclass[a4paper,oneside,reqno]{amsart}

\usepackage[T1]{fontenc}
\usepackage[utf8]{inputenc}
\usepackage{amsfonts}
\usepackage{mathtools}
\usepackage[hidelinks]{hyperref}
\usepackage{amsthm,thmtools,amsmath,amssymb}
\usepackage{tikz}
\usepackage{enumerate, enumitem}
\usepackage{bbm}
\usepackage{microtype}
\usepackage{a4wide}
\usetikzlibrary{patterns}
\tikzset{every node/.style={circle,inner sep=0pt,minimum size=3.5pt,fill=black,draw}}

\newcommand{\mathbold}[1]{\mathbb{#1}}
\renewcommand{\emptyset}{\varnothing}

\newcommand{\NN}{\mathbold{N}}

\newcommand{\cF}{\mathcal{F}}
\newcommand{\cX}{\mathcal{X}}
\newcommand{\att}{\hookrightarrow}

\newcommand{\p}{\mathbf p}

\newcommand{\Sym}{\mathrm{Sym}}
\renewcommand{\Pr}{\mathbb{P}}

\declaretheorem{theorem}
\declaretheorem[sibling=theorem]{lemma}

\declaretheorem[sibling=theorem]{corollary}
\declaretheorem[sibling=theorem]{proposition}
\declaretheorem[style=definition,sibling=theorem]{definition}
\declaretheorem[style=definition,sibling=theorem]{example}

\DeclareMathOperator{\e}{\mathbb{E}}
\DeclareMathOperator{\cut}{Cut}
\DeclareMathOperator{\Aut}{Aut}
\newcommand{\eps}{\varepsilon}
\newcommand{\ol}[1]{\overline{#1}}
\newcommand{\ul}[1]{\underline{#1}}

\newcommand{\mytriangle}[3]{
  \begin{scope}[shift={#1}, rotate=#2]
    \coordinate (a) at (0,0);
    \coordinate (b) at (1,0);
    \coordinate (c) at (0.5,0.866025);
    \draw[fill opacity=0.3,#3] (a) -- (b) -- (c) -- cycle;
    \node at (a) {};
    \node at (b) {};
    \node at (c) {};
  \end{scope}
}

\newcommand{\myquad}[3]{
  \begin{scope}[shift={#1}, rotate=#2]
    \coordinate (a) at (0,0);
    \coordinate (b) at (1,0);
    \coordinate (c) at (1,1);
    \coordinate (d) at (0,1);
    \draw[fill opacity=0.3,#3] (a) -- (b) -- (c) -- (d) -- cycle;
    \node at (a) {};
    \node at (b) {};
    \node at (c) {};
    \node at (d) {};
  \end{scope}
}

\begin{document}

  \title{On the probability of nonexistence in binomial subsets}
\author{Frank Mousset}
\address{Frank Mousset, School of Mathematical Sciences, Tel Aviv University, Tel Aviv 6997801, Israel}
\email{moussetfrank@gmail.com}

\author{Andreas Noever}
\address{Andreas Noever, Department of Computer Science, ETH Zurich, 8092 Zurich, Switzerland}
\email{andreas.noever@gmail.com }

\author{Konstantinos Panagiotou}
\address
{Konstantinos Panagiotou, Institute of Mathematics,
    University of Munich,
    D-80333 Munich, Germany}
\email{kpanagio@math.lmu.de}

\author{Wojciech Samotij}
\address{Wojciech Samotij, School of Mathematical Sciences, Tel Aviv University, Tel Aviv 6997801, Israel}
\email{samotij@tauex.tau.ac.il}

\thanks{
  Research supported in part by the Israel Science Foundation grants 1147/14
  (FM, WS), 1028/16 (FM) and 1147/14 (FM), the European Research Council (ERC)
  under the European Union’s Horizon 2020 research and innovation program
  (grant agreement no. 772606) (KP), and the ERC Starting Grant 633509 (FM).
}
\date{\today}

  \begin{abstract}
    Given a hypergraph $\Gamma=(\Omega,\cX)$ and a sequence $\p
    = (p_\omega)_{\omega\in \Omega}$ of values in $(0,1)$,
    let $\Omega_\p$ be the random
    subset of $\Omega$ obtained by keeping every vertex $\omega$ independently
    with probability $p_\omega$. We investigate the general question of
    deriving fine (asymptotic) estimates for the probability that $\Omega_\p$
    is an independent set in $\Gamma$, which is an omnipresent problem in
    probabilistic combinatorics. Our main result provides a sequence of upper
    and lower bounds on this probability, each of which can be evaluated
    explicitly in terms of the joint cumulants of small sets of edge indicator
    random variables. Under certain natural conditions, these upper and lower
    bounds coincide asymptotically, thus giving the precise asymptotics of the
    probability in question. We demonstrate the applicability of our results
    with two concrete examples: subgraph containment in random (hyper)graphs
    and arithmetic progressions in random subsets of the integers.
  \end{abstract}

  \maketitle

\section{Introduction}

Let $\Gamma = (\Omega,\cX)$ be a hypergraph and,
given a sequence $\p = (p_\omega)_{\omega\in \Omega}\in (0,1)^\Omega$,
let $\Omega_\p$ be a random subset of $\Omega$ formed by including every $\omega\in\Omega$
independently with probability $p_\omega$. 
What is the probability that
$\Omega_\p$ is an independent set in $\Gamma$? This very general question arises
in many different settings.

\begin{example}
  \label{j:ex:F-free} Let $F$ be a graph and let $n$ be a positive integer. 
  Define $\Omega$ as the edge set $E(K_n) = \binom{[n]}{2}$ of the complete graph
  with vertex set $[n]=\{1, \dotsc, n\}$ and let $\cX$ be the collection of
  the edge sets of all copies of $F$ in~$K_n$. Fix some $p \in (0,1)$ and define
  $\p$ by setting $p_\omega=p$ for every $\omega\in\Omega$. Then we are asking
  for the probability that the Erd\H{o}s--Rényi random graph $G_{n,p}$ is $F$-free, that
  is, does not contain $F$ as a (not necessarily induced) subgraph.
\end{example}

\begin{example}
  \label{j:ex:r-AP}
  An arithmetic progression of length $r\in \NN$ (an \emph{$r$-AP} for short)
  is a subset of the integers of the form $\{a+kb : k\in [r]\}$, where $b\neq
  0$. Let $\Omega = [n]$ and let $\cX$ be the set of all $r$-APs in $[n]$.
  Given $p \in (0,1)$, we define $\p$ by setting $p_\omega=p$ for every
  $\omega\in\Omega$. 
   Then we are asking for the
   probability that the $p$-random subset $[n]_p$ of $[n]$ is $r$-AP-free.
\end{example}

\begin{example}
  \label{j:ex:gen-pos}
  Let $\Omega$ be a finite set of points in the plane. Include a triple
  $\{i,j,k\}$ in $\cX$ if the points $i,j,k$ lie on a common line. Now we are
  asking for the probability that the random subset $\Omega_\p$ of points
  is in general position.
\end{example}

It is not hard to find other natural examples that provide further motivation
for studying this question. It is convenient to introduce some
notation. Given
$\Gamma=(\Omega,\cX)$ and $\p\in (0,1)^\Omega$, we shall fix an (abritrary) ordering of the
elements of $\cX$ as $\gamma_1,\dotsc,\gamma_N$.
We then let $X_i$ denote the indicator variable of the event that
$\gamma_i\subseteq \Omega_\p$ and set $X = X_1 + \dotsb + X_N$. Thus, $X$
counts the number of edges of $\Gamma$ that are fully contained in $\Omega_\p$ and our goal
is to compute the probability that $X=0$. Of course, these notations all depend
on the given pair $(\Gamma,\p)$, but we shall always suppress this
dependence as it will be clear from the context.

%
Most of the time, we will be interested in the case where $\Gamma = \Gamma(n)$
and $\p=\p(n)$ (and hence also $X = X(n)$) depend on some parameter $n$
that tends to infinity and ask:
\begin{center}
  \textit{What are the asymptotics of the probability $\Pr[X=0]$ as
  $n\to\infty$?}
\end{center}
The above question can also be viewed as a computational
problem: we want to derive closed formulas that are asymptotic to $\Pr[X=0]$,
at least for various ranges of the density parameter $\p$.

For technical convenience, we shall exclude the border case
where $p_\omega\in \{0,1\}$ for some $\omega$. That case can always be
addressed by changing $\Gamma$ or by a continuity argument.

%


\subsection{The Harris and Janson inequalities}

The main reason why computing $\Pr[X=0]$ is challenging is that the random
variables $X_1,\dotsc,X_N$ are usually not independent. However, this is not
to say that there is no structure at all: each random variable $X_i$ is a
non-decreasing function on the product space $\{0,1\}^\Omega$. An important
inequality that applies in this case is the \emph{Harris inequality}:

\begin{theorem}[Harris inequality~\cite{Ha60}]
  Let $\Omega$ be a finite set and let $X$ and $Y$ be random variables defined
  on a product probability space  over $\{0,1\}^\Omega$. If $X$ and $Y$ are
  both non-decreasing (or non-increasing), then
  \[ \e[XY] \geq \e[X] \e[Y]. \]
  If $X$ is non-decreasing and $Y$ is non-increasing, then
  \[ \e[XY] \leq \e[X] \e[Y]. \]
\end{theorem}

In our setting, for every $V\subseteq [N]$, the random variable
$\prod_{i \in V} (1-X_i)$ is non-increasing, so we easily deduce
from Harris's inequality that
\begin{equation}
  \label{j:eq:harris}
  \Pr[X=0] = \e\left[\prod_{i\in [N]} (1-X_i)\right] \geq
  \prod_{i\in [N]} (1- \e[X_i]).
\end{equation}
Note that \eqref{j:eq:harris} would be true with equality if 
$X_1,\dotsc,X_N$ were independent. An upper bound on
$\Pr[X=0]$ is given by \emph{Janson's inequality}, which states that
the reverse of~\eqref{j:eq:harris} holds up to a multiplicative error term that
is an explicit function of the pairwise dependencies between the 
indicator random variables $X_1,\dotsc,X_N$. Formally, we write $i\sim
j$ if $i\neq j$ and
$\gamma_i\cap \gamma_j \neq \emptyset$, and define the sum of joint moments
\begin{equation}\label{eq:delta2}
  \Delta_2 = \sum_{\substack{i\sim j}} \e[X_iX_j].
\end{equation}

\begin{theorem}[Janson's inequality~{\cite{AlSp16, JaLuRu00}}]
  \label{thm:janson}
  For all $\Gamma$ and $\p$ as above,
  \[ \Pr[X = 0] \leq \exp\big(-\e[X] + \Delta_2\big). \]
\end{theorem}

To compare this with~\eqref{j:eq:harris}, we will now assume
that the individual probabilities of $X_i=1$ are not too large, say
$\e[X_i] \leq 1-\eps$ for some positive constant $\eps$. In this case, we may use the
inequality $1 - x \geq 
\exp(- x- x^2/\eps)$ for $x\in
[0, 1-\eps]$ to obtain from~\eqref{j:eq:harris}
\begin{equation}
  \Pr[X=0] 
  \geq \prod_{i\in [N]} (1- \e[X_i])
  \geq \exp(-\e[X] - \delta_1/\eps),
\end{equation}
where
\begin{equation}
\label{eq:delta_1}
  \delta_1 = \sum_{i\in [N]} \e[X_i]^2.
\end{equation}
Combining this lower bound with the upper bound given by Janson's inequality, we get the
approximation
\begin{equation}\label{j:eq:combi}
  \Pr[X=0] = 
  \exp\big(-\e[X] + O(\delta_1+\Delta_2)\big).
\end{equation}

If $\delta_1+\Delta_2=o(1)$, then
\eqref{j:eq:combi} gives the correct
asymptotics of $\Pr[X=0]$. The condition $\Delta_2
=o(1)$ in particular requires that
the pairwise correlations between the indicator variables
$X_i$ vanish asymptotically in a well-defined sense. This
rather strict requirement is not satisfied in many natural settings, including
the ones
presented in Examples~\ref{j:ex:F-free}--\ref{j:ex:gen-pos} for certain
choices of $p$. It is therefore an important
question to obtain better approximations of $\Pr[X = 0]$ in cases
when the pairwise dependencies among the $X_i$ are not negligible.
 This is the starting point
of our investigations.

\subsection{Triangles in random graphs}

Even though our results will be phrased in the general framework introduced
above and are thus widely applicable, we believe that it is useful to keep in
mind the following well-studied instance of the problem that will serve as a
guiding example.

\begin{example}\label{j:ex:tria}
  Assume $p=p(n)\in (0,1)$ and 
  let $X=X(n)$ denote the number of triangles in $G_{n,p}$, as in
  Example~\ref{j:ex:F-free} with $F=K_3$.
  Since each triangle has three edges, we have $\e[X_i] = p^3$ for all
  $i$. Thus $\e[X] = \binom{n}{3}p^3$ and $\delta_1 = O(n^3p^6)$.
  Moreover, we have $\Delta_2 = O(n^4p^5)$, because if two distinct triangles intersect,
  then their union is the graph with $4$ vertices and $5$ edges.
  Thus~\eqref{j:eq:combi} implies that as long as $p=o(n^{-4/5})$, we have
  \[
    \Pr[X=0] = \exp\big(-n^3p^3/6 + o(1)\big).
  \]
  Extending this result,
  Wormald~\cite{Wo96} and
  later Stark and Wormald~\cite{StWo} obtained asymptotic expressions for
  $\Pr[X=0]$ even when $p=\Omega(n^{-4/5})$ and
  thus~\eqref{j:eq:combi} no longer gives an asymptotic bound. In particular,
  it was shown by Stark and Wormald in~\cite{StWo} that if $p=o(n^{-7/11})$,
  then
  \[ \Pr[X=0] = \exp\Big(- \frac{n^3p^3}{6} +
  \frac{n^4p^5}{4} - \frac{7n^5p^7}{12} + \frac{n^2p^3}{2} - \frac{3n^4p^6}{8}
  + \frac{27n^6p^9}{16}+o(1)\Big). \]
  One goal of this paper is to give a simple interpretation of the individual
  terms in this formula. Indeed, we will formulate a general result from which
  the above formula may be obtained by a few short calculations. More
  precisely, we will prove a generalisation of~\eqref{j:eq:combi} that takes
  into account the $k$-wise dependencies between the variables $X_i$
  for all $k \geq 2$.
\end{example}

\subsection{Joint cumulants, clusters, dependency graphs}

Let $A = \{Z_1,\dotsc,Z_m\}$ be a finite set of real-valued random variables. The \emph{joint moment} of the
variables in $A$ is
\begin{equation}
  \Delta(A) = \e[Z_1 \dotsm Z_m].
\end{equation}
The \emph{joint cumulant} of the variables is
\begin{equation}\label{j:eq:kappa}
  \kappa(A) = \sum_{\pi\in \Pi(A)} (|\pi|-1)!(-1)^{|\pi|-1} \prod_{P \in \pi}
  \Delta(P),
\end{equation}
where $\Pi(A)$ denotes the set of all partitions of $A$ into non-empty
sets. In particular,
\[
\begin{split}
  \kappa(\{X\}) & = \e[X],\\
  \kappa(\{X,Y\}) & = \e[XY] - \e[X]\e[Y],\\
  \kappa(\{X,Y,Z\}) & = \e[XYZ] - \e[X]\e[YZ] - \e[Y]\e[XZ] - \e[Z]\e[XY]\\
&\qquad + 2 \e[X]\e[Y]\e[Z].
\end{split}
\]
The joint cumulant $\kappa(A)$ can be regarded as a measure of the mutual
dependence of the variables in $A$. For example, $\kappa(\{X,Y\})$ is simply
the covariance of $X$ and $Y$, and so $\kappa(\{X,Y\})=0$ if $X$ and
$Y$ are independent.  More generally, the following holds.

\begin{proposition}
  \label{prop:cumulant-ind}
  Let $A$ be a finite set of real-valued random variables. If $A$ can be
  partitioned into two subsets $A_1$ and $A_2$ such that all variables in
  $A_1$ are independent of all variables in $A_2$, then $\kappa(A) = 0$.
\end{proposition}

In fact, Proposition~\ref{prop:cumulant-ind} remains valid when one replaces
the independence assumption with the weaker assumption that $\Delta(B_1 \cup
B_2) = \Delta(B_1) \Delta(B_2)$ for all $B_1 \subseteq A_1$ and $B_2 \subseteq
A_2$. An elegant proof of Proposition~\ref{prop:cumulant-ind} can be found
in~\cite{AhUsPi}. The proposition motivates the definition of the following
notion.

\begin{definition}[decomposable, cluster]
  \label{definition:cluster}
  A set $A$ of random variables is \emph{decomposable} if there exists
  a partition $A= A_1\cup A_2$ such that the variables in $A_1$ are independent of the variables in $A_2$.
  A~non-decomposable set is also called a \emph{cluster}.
\end{definition}

In our setting, the notion of a cluster has a natural combinatorial
interpretation. Given $\Gamma=(\Omega,\cX)$ and $\p\in (0,1)^\Omega$, we define the
\emph{dependency graph} $G_\Gamma$ as the graph on the vertex set $[N]$ whose
edges are all pairs $\{i,j\}$ such that $i\sim j$, that is, $\gamma_i\cap \gamma_j\neq \emptyset$. It is then clear
that a set $V\subseteq [N]$ induces a connected subgraph in $G_\Gamma$
if and only if the set of random variables $\{X_i:
i\in V\}$ is a cluster (this is one reason why it
is convenient to assume $p_\omega\notin\{0,1\}$ for all $\omega\in \Omega$). In particular, the joint cumulant
$\kappa(\{X_i: i\in V\})$ vanishes unless
$G_\Gamma[V]$ is connected.

Motivated by this, we shall write $\mathcal C_k$ for the collection
of all $k$-element subsets $V\subseteq [N]$ such that $G_\Gamma[V]$ is
connected, and define
\begin{equation}\label{eq:kappadelta}
  \kappa_k =\sum_{V\in \mathcal C_k}\kappa(\{X_i: i\in V\})
  \quad\text{and}\quad
  \Delta_k =\sum_{V\in \mathcal C_k}\Delta(\{X_i: i\in V\}).
\end{equation}
Note that this definition of $\Delta_k$ is consistent with the definition of
$\Delta_2$ given by~\eqref{eq:delta2}. 
Moreover, it follows from~\eqref{j:eq:kappa} and Harris's inequality that
$|\kappa_k|\leq K_k\Delta_k$ for some $K_k$ depending only on~$k$.

\subsection{Our main result} 
Let $\Gamma=(\Omega,\cX)$ and $\p\in (0,1)^\Omega$ be as above.
Given a subset $V\subseteq [N]$, we write
\[
  \partial(V) = N_{G_\Gamma}(V)\setminus V
\]
for the external neighbourhood of $V$ in the dependency graph and let
\[
  \lambda(V) = \sum_{i\in \partial(V)} \e[X_i \mid \prod_{j\in V} X_j = 1]
\]
be the expected number of external neighbours $i$ of $V$ in the dependency
graph such that
$\gamma_i\subseteq \Omega_\p$, conditioned on $\gamma_j\subseteq \Omega_\p$ for
all $j\in V$. For all $k\in\NN$, we define
\[
  \Lambda_k(\Gamma,\p) = \max\big\{\lambda(V) : V\subseteq [N]
  \text{ and }1\leq |V|\leq k\big\}.
\]
It can be intuitively helpful to think of $\Lambda_k(\Gamma,\p)$ as a measure
of (non)expansion of the dependency graph $G_\Gamma$.

\begin{theorem}\label{j:thm:sparse}
  For every $n\in \NN$, let $\Gamma(n)=(\Omega(n),\cX(n))$ be a hypergraph and
  let $\p(n) \in (0,1)^{\Omega(n)}$. Assume that
  for every constant $k\in \NN$,
  \[ \lim_{n\to \infty} \max_{\omega\in \Omega(n)}p_\omega(n) = 0
  \quad \text{and} \quad
  \limsup_{n\to\infty} \Lambda_k(\Gamma(n),\p(n)) < \infty.
  \] 
  Let $X(n)$ denote the number of edges of $\Gamma(n)$ that are fully
  contained in $\Omega(n)_{\p(n)}$.
  Then, for every constant $k\in \NN$,
  \[ \Pr[X(n)=0] = \exp\big(- \kappa_1 + \kappa_2-\dotsb +(-1)^k \kappa_k
  +O(\delta_1+ \Delta_{k+1})\big) \]
  as $n\to\infty$,
  where $\delta_1$, $\kappa_1,\dotsc,\kappa_k$, and
  $\Delta_{k+1}$ are defined as above.
\end{theorem}

The condition $\max{\{p_\omega(n): \omega\in\Omega(n)\}}=o(1)$
implies 
$\kappa_k = \Delta_k + o(\Delta_k)$ for every fixed $k$,
as can
be seen from the definition \eqref{j:eq:kappa} of $\kappa_k$. In such cases, the first-order
behaviour of $\kappa_k$ is thus given by $\Delta_k$. However, this does
\emph{not} mean that we can then replace $\kappa_i$ by $\Delta_i$ in the
formula for $\Pr[X(n)=0]$ given by Theorem~\ref{j:thm:sparse}, because the
lower-order terms in the $\kappa_i$ can be non-negligible, see e.g.\ the proof
of Corollary~\ref{cor:k3} below.

The fact that $\kappa_1 = \e[X]$ shows that the case $k=1$ of
Theorem~\ref{j:thm:sparse} gives (a slight weakening of)
Janson's inequality~\eqref{j:eq:combi}. Unlike \eqref{j:eq:combi},
Theorem~\ref{j:thm:sparse} requires the additional assumptions $\max_{\omega\in
\Omega(n)}p_\omega(n)= o(1)$ and $\Lambda_k(\Gamma(n),\p(n)) = O(1)$
for all constant $k$.
Both conditions are perhaps not
strictly necessary. As we will see further below, the latter condition
implies that $\Delta_{k+1} = O(\Delta_k)$ for all constant $k$, which gives at
least an indication of the type of assumption that is involved.

It is natural to ask under which conditions Theorem~\ref{j:thm:sparse} can give
asymptotically sharp bounds. While computing
the first error term $\delta_1$ is generally
straightforward, it is not so obvious how one should estimate
$\Delta_{k+1}$.
Here we will focus on the rather common situation where each edge of
$\Gamma(n)$ has bounded size and there
is some $p(n) \in (0,1)$ such that $\p_\omega(n) = p(n)$ for all $\omega\in
\Omega(n)$. We then write simply $\Omega_p$ instead of $\Omega_\p$. This is the situation that we encounter in all of our applications.

For every $\Omega'\subseteq \Omega$,
define
the $j$-th codegree of $\Omega'$ by
\[d_j(\Omega') = |\{\gamma\in \cX : \Omega'\subseteq
\gamma \text{ and }|\gamma| = |\Omega'|+j\}|,\]
and let
\[ D(\Gamma,p) = \max_{j\geq 1} 
\max_{\emptyset\neq \Omega'\subseteq \Omega} d_j(\Omega')p^j; \] one
can think of this as a weighted maximum codegree of $\Gamma$. The
following is a specialised version of Theorem~\ref{j:thm:sparse} that gives an
easily verifiable condition ensuring $\Delta_{k+1} = o(1)$ for some
constant $k$.

\begin{theorem}\label{j:thm:bounded}
  Let $r$ be a fixed positive integer. For every $n\in \NN$, let $\Gamma(n) = (\Omega(n),\cX(n))$ be a hypergraph whose edges
  all have size at most $r$
  and let $p(n)$ be a real number in
  $(0,1)$. Assume
  \[ \lim_{n\to\infty} p(n)  = 0\quad\text{and}\quad 
  \limsup_{n\to\infty}D(\Gamma(n),p(n)) <\infty.\] Let
  $X(n)$ denote the number of edges of $\Gamma(n)$ that are fully contained in
  $\Omega(n)_{p(n)}$. Then, for every constant $k\in \NN$,
  \[ \Pr[X(n)=0] = \exp\big(- \kappa_1 + \kappa_2-\dotsb +(-1)^k \kappa_k
  +O(\delta_1+ \Delta_{k+1})\big) \]
  as $n\to\infty$,
  where $\delta_1$, $\kappa_1,\dotsc,\kappa_k$, and $\Delta_{k+1}$
  are defined as above.
  
  Moreover, if $D(\Gamma(n),p(n)) \leq |\Omega(n)|^{-\eps}$ 
  for some positive $\eps$, then there is a positive integer $k = k(\eps,r)$
  such that $\Delta_{k+1} =o(1)$.
\end{theorem}

Let us briefly illustrate the applicability of this result by considering
again the example of 
triangle-free random graphs.

\begin{example}[continues=j:ex:tria]
  The hypergraph $\Gamma$ of triangles in $K_n$ is $3$-uniform, so we can
  choose $r=3$ in Theorem~\ref{j:thm:bounded}. One easily verifies
  that
  $D(\Gamma,p) \leq p+ np^2$.
  We recall from our earlier discussion that $\delta_1 \leq n^3p^6$.
  Therefore Theorem~\ref{j:thm:bounded}
  implies that
  for every fixed positive integer $k$ and all $p=o(n^{-1/2})$, we have
  \[ \Pr[X = 0] = \exp\big(-\kappa_1 + \kappa_2-\dotsb +(-1)^k \kappa_k
  +O(\Delta_{k+1}) + o(1)\big). \]
  Moreover, if $p \leq n^{-1/2-\eps}$ for
  some positive constant $\eps$, then there exists a constant $k$ such that
  \[ \Pr[X = 0] = \exp\big(-\kappa_1 + \kappa_2-\dotsb +(-1)^k \kappa_k
  +o(1)\big), \]
  i.e., the asymptotics of $\Pr[X = 0]$ are given by a finite formula
  that we could in principle compute by analysing the finitely many
  possible `shapes' of clusters formed by at most $k$ triangles in $K_n$.
\end{example}

We shall derive both of the above theorems from a more general result,
Theorem~\ref{j:thm:main} below, which
has the advantage that it can
be applied in certain non-sparse settings. Its disadvantage lies
in the fact that
the error terms are somewhat less transparent.
For a set $A$ of random variables, we define
\[
\delta(A) = \Delta(A)\cdot \max{\{\e[X] : X \in A\}}.
\]
Given $k\in\NN$, we set
\begin{equation}\label{eq:delta_k}
\delta_k = \sum_{V\in \mathcal C_k}\delta(\{X_i : i\in V\}),
\end{equation}
analogously to \eqref{eq:kappadelta},
and
\begin{equation}
  \rho_k = \max_{\substack{V\subseteq [N]\\1\leq |V|\leq k}}
\Pr[X_i=1\text{ for some $i\in V \cup \partial(V)$}].
\end{equation}
Observe that this definition of $\delta_k$ generalises~\eqref{eq:delta_1}.

\begin{theorem}\label{j:thm:main}
  For every $k\in\NN$ and $\eps>0$, there is a
  $K=K(k,\eps)$ such that the following holds.
  Let $\Gamma=(\Omega,\cX)$ be a hypergraph
  and let $\p \in (0,1)^\Omega$.
  If $\rho_{k+1}\leq 1-\eps$, then
  \[
    \big|\log\Pr[X=0] + \kappa_1 - \kappa_2 + \kappa_3 - \dotsb + (-1)^{k+1}
    \kappa_k\big| \leq K \cdot (\delta_{1,K} + \Delta_{k+1,K}),
  \]
  where
  \[ \delta_{1,K} = \sum_{i=1}^K \delta_i \qquad\text{and}\qquad
  \Delta_{k+1,K} = \sum_{i=k+1}^K\Delta_i. \]
\end{theorem}

We will derive Theorems~\ref{j:thm:sparse} and~\ref{j:thm:bounded} from
Theorem~\ref{j:thm:main} in Section~\ref{j:sec:discussion}.  The proof of
Theorem~\ref{j:thm:main}, which is the main part of this paper, will be
presented in Section~\ref{j:sec:proof}.

\subsection{Application: random graphs and hypergraphs}
\label{ssec:graphs}

A fundamental question studied by the random graphs community, raised already
in the seminal paper of Erd\H{o}s and R\'enyi~\cite{ErRe60}, is to determine
the probability that $G_{n,p}$ contains no copies of a given `forbidden' graph
$F$ (as in Example~\ref{j:ex:F-free}). The classical result of
Bollob\'as~\cite{Bo81}, proved independently by Karo\'nski and
Ruci\'nski~\cite{KaRu83}, determines this probability asymptotically for every
strictly balanced\footnote{A graph $F$ is strictly balanced if $e_F/v_F >
e_H/v_H$ for every proper non-empty subgraph $H$ of $F$.} $F$, but only for $p$
such that the expected number of copies of $F$ in $G_{n,p}$ is constant. (In
the case when $F$ is a tree or a cycle, this was done earlier by Erd\H{o}s and
R\'enyi~\cite{ErRe60} and in the case when $F$ is a complete graph, by
Sch\"urger~\cite{Sc79}.) It was later proved by Frieze~\cite{Fr92} that the
same estimate remains valid as long as the expected number of copies of $F$ in
$G_{n,p}$ is $o(n^\eps)$ for some positive constant $\eps$ that depends only on
$F$. Wormald~\cite{Wo96} and later Stark and Wormald~\cite{StWo} obtained
asymptotic formulas for significantly larger ranges of $p$ in the special case
where $F$ is a triangle.
Prior to those papers and the present work, the strongest result of this form
(i.e., determining the probability of being $F$-free asymptotically) for a
general graph $F$ followed from Harris's and Janson's inequalities,
see~\eqref{j:eq:combi}. Finally, we remark that for several special graphs $F$,
the probability that $G_{n,p}$ is $F$-free can be computed very precisely either when
$p=1/2$ or, in some cases, even for all sufficiently large $p=o(1)$ using the
known precise structural characterisations of $F$-free graphs,
see~\cite{BaMoSaWa16, HuPrSt93, OsPrTa03, PrSt92,promel1996asymptotic}.

Using Theorem~\ref{j:thm:bounded}, we can 
answer this question for a large class of graphs and a wide range of densities.
We will take a rather general point of view
and consider the analogous problem in random $r$-uniform hypergraphs,
where instead of just avoiding a single graph $F$, our goal is to avoid
every graph in some finite
family $\cF$.
Let $G^{(r)}_{n,p}$ denote the random $r$-uniform
hypergraph ($r$-graph for short)
on $n$ vertices containing every possible edge ($r$-element subset of the
vertices) with probability $p$, independently of other edges. In particular,
$G_{n,p}^{(2)}$ is simply the binomial random graph $G_{n,p}$. Given a family
$\mathcal F = \{F_1, \dotsc, F_t\}$ of $r$-graphs, we consider the problem of
determining the probability that $G^{(r)}_{n,p}$ is \emph{$\mathcal F$-free},
that is, it simultaneously avoids
all copies of all $r$-graphs in $\mathcal F$.

Since removing isomorphic duplicates from $\cF$
does not affect the probability that we are
interested in, we can assume that the $r$-graphs in $\mathcal F$ are pairwise
non-isomorphic. Similarly, we can assume that no hypergraph in $\cF$ has isolated vertices.

We encode this problem in a hypergraph $\Gamma = (\Omega,\cX)$ by proceeding
similarly as we did in Example~\ref{j:ex:F-free}. That is, we let $\Omega
=\binom{[n]}{r}$ be the edge set of $K_n^{(r)}$, the complete $r$-graph with
vertex set $[n]$, and we let $\cX$ be the collection of edge sets of
subhypergraphs of $K_n^{(r)}$ that are isomorphic to one of the $r$-graphs in
$\mathcal F$. The probability that $G_{n,p}$ is $\cF$-free is then precisely
the probability that the $p$-random subset $\Omega_p$ contains no edges of
$\Gamma$.

Note that the maximal size of an edge in $\Gamma$ is bounded
from above by the largest number of edges of an $r$-graph in $\mathcal F$,
which does not depend on~$n$. By applying Theorem~\ref{j:thm:bounded} to this
hypergraph, we can therefore get the asymptotics for the probability that
$G_{n,p}^{(r)}$ is $\mathcal F$-free in a certain range of $p$.
To quantify this range, 
given an $r$-graph $F$, define
\[
  m_*(F) = \min \left\{ \frac{e_F-e_H}{v_F-v_H} : \text{$H \subseteq F$ with $v_H < v_F$ and $e_H > 0$} \right\},
\]
where we use the convention $\min\emptyset=\infty$
and where $v_G$ and $e_G$ denote, respectively, the numbers of vertices and edges in
a (hyper)graph $G$. For a family $\mathcal F$ of $r$-graphs, we then set
\[
  m_*(\mathcal F) = \min{\{ m_*(F) : F\in \mathcal F\}}
  \quad
  \text{and}
  \quad
  d(\mathcal F) = \min{\{e_F/v_F : F\in\mathcal F\}}.
\]
It is easy to see that $\delta_1 \leq
|\mathcal F|\cdot \max{\{n^{v_F}p^{2e_F}: F\in \mathcal F\}}$ and thus $\delta_1=o(1)$
if $np^{2d(\mathcal F)} = o(1)$. Moreover, for any non-empty set $\Omega'$ of edges in $K^{(r)}_n$
whose union forms an $r$-graph $H$ with $e_H>0$ edges,
we have
\[
  \max_{j\geq 1} d_j(\Omega')p^j = O\big(\max{\{ n^{v_F-v_H}p^{e_F-e_H} : H\subseteq F\in
  \mathcal F\text{
    and } v_H< v_F\}
}\big).\]
It follows that $D(\Gamma,p) = (np^{m_*(\mathcal F)})^{\Theta(1)}$.
Theorem~\ref{j:thm:bounded} then immediately implies the following result.

\begin{corollary}
  \label{cor:graphs}
  Let $\mathcal F$ be a finite family of $r$-uniform hypergraphs
  and assume that $p=p(n)\in (0,1)$ satisfies
  \begin{equation} \label{eq:graphs-assumption}
    np^{m_*(\mathcal F)}=o(1)
    \qquad \text{and} \qquad
    np^{2d(\mathcal F)} = o(1).
  \end{equation}
  Then, for every constant $k\in \NN$, we have
  \[ \Pr\left[\text{$G_{n,p}^{(r)}$ is $\mathcal F$-free}\right] =
  \exp\big(-\kappa_1+\kappa_2-\dotsb+(-1)^k\kappa_k +O(\Delta_{k+1})+o(1)\big)
  \]
  as $n\to\infty$.
  Moreover, if 
  $np^{m_*(\mathcal F)}\leq
  n^{-\eps}$ for some positive $\eps$, then
  there is a positive integer $k = k(\eps,\mathcal F)$ such
  that $\Delta_{k+1}=o(1)$.
\end{corollary}

The conditions in \eqref{eq:graphs-assumption} can be further simplified under
certain natural assumptions on the family $\mathcal F$. Recall that the
\emph{$r$-density} of an $r$-graph $F$ with at least two edges is
\[
  m_r(F) = \max \left\{ \frac{e_H-1}{v_H-r} : \text{$H \subseteq F$ with $e_H > 1$} \right\}
\]
and that $F$ is \emph{$r$-balanced} if this maximum is achieved with $H = F$,
that is, if $m_r(F) = (e_F-1)/(v_F-r)$. 
Observe that for every $F$ with at
least two edges, we have
\[
  m_r(F) \ge \frac{e_F-1}{v_F-r} \ge m_*(F).
\]
We claim that if $F$ is $r$-balanced, then in fact $m_r(F) = m_*(F)$. Indeed,
writing $\alpha_K = (e_K-1)/(v_K-r)$,
we see that for every $H \subseteq F$ with $v_H < v_F$ and $e_H > 1$,
\[
  \frac{e_F - e_H}{v_F - v_H} = \frac {\alpha_F(v_F - r) -
\alpha_H(v_H-r)}{(v_F - r) - (v_H-r)} \ge 
  m_r(F),
\]
since $m_r(F) = \alpha_F \ge \alpha_H$ (as $F$ is $r$-balanced) and this inequality
continues to hold if $e_H = 1$. Thus $m_*(F) \ge
m_r(F)$ and so $m_*(F) = m_r(F)$.

Another simplification is possible in the important special case $r=2$. In
this case, the second condition in~\eqref{eq:graphs-assumption} follows from
the first condition, since $2e_F/v_F \ge (e_F-1)/(v_F-2)$ for every graph $F$
and consequently $m_*(\mathcal F) \le 2d(\mathcal F)$ for every family of
graphs $\mathcal F$.

\begin{corollary}
  \label{cor:balanced-graphs}
  Let $\mathcal F$ be a finite family of $2$-balanced graphs
  with at least two edges each and let
  $p=p(n)\in(0,1)$ be such that $p = o(n^{-1/m_2(F)})$ for every $F \in
  \mathcal F$. Then, for every fixed $k\in \NN$, we have
  \[
    \Pr\left[\text{$G_{n,p}$ is $\mathcal F$-free}\right] =
    \exp\big(-\kappa_1+\kappa_2-\dotsb+(-1)^k\kappa_k
    +O(\Delta_{k+1})+o(1)\big).
  \]
  as $n\to\infty$.
  Moreover, if 
  $p\leq n^{-1/m_2(F)-\eps}$ for
  some positive $\eps$ and all $F\in\mathcal F$, then
  there is a positive integer $k = k(\eps,\mathcal F)$ such
  that $\Delta_{k+1}=o(1)$.
\end{corollary}

Of course, neither Corollary~\ref{cor:graphs} nor
Corollary~\ref{cor:balanced-graphs} would be particularly useful if one could
not compute the values $\kappa_k$ for at least several small integers~$k$.
In Section~\ref{j:sec:applications}, we outline a general approach for doing
so and perform the
calculations for two special cases.

\begin{corollary}
  \label{cor:k3c4}
  If $p = o(n^{-4/5})$, then the probability that $G_{n,p}$ is simultaneously
  $K_3$-free and $C_4$-free is asymptotically
  \[
    \exp\Big(- \frac{n^3p^3}{6} - \frac{n^4p^4}{8} + \frac{n^6p^7}{4} + \frac{2n^5p^6}{3}\Big).
  \]
\end{corollary}

\begin{corollary}
  \label{cor:k3}
  If $p = o(n^{-7/11})$, then the probability that $G_{n,p}$ is triangle-free is asymptotically
  \[
    \exp\Big(- \frac{n^3p^3}{6} + \frac{n^4p^5}{4} - \frac{7n^5p^7}{12} + \frac{n^2p^3}{2} - \frac{3n^4p^6}{8} + \frac{27n^6p^9}{16} \Big).
  \]
\end{corollary}

As mentioned above, Corollary~\ref{cor:k3} was obtained independently by Stark
and Wormald~\cite{StWo}. It extends a result of
Wormald~\cite{Wo96} that applies to a smaller range of $p$. However, the
derivation of Corollary~\ref{cor:k3} from Theorem~\ref{j:thm:bounded} is very
short compared to the proofs in~\cite{StWo} and~\cite{Wo96}.

\subsection{Application: arithmetic progressions}
\label{ssec:aps}

As a second application, we will estimate the probability that $[n]_p$, the
$p$-random subset of $[n]$, is $r$-AP-free, i.e., does not contain
any arithmetic progression of length $r$.
As in Example~\ref{j:ex:r-AP}, we encode this problem in
the hypergraph $\Gamma=(\Omega,\cX)$ on $\Omega = [n]$ whose edge set is the collection
$\cX$ of $r$-APs in $[n]$.

Since any two distinct integers are contained at most $\binom{r}{2} = O(1)$ common
$r$-APs, it is easy to see that $\delta_1 = O(n^2p^{2r})$
and
$D(\Gamma,p) = O(p+np^{r-1})$. Therefore, Theorem~\ref{j:thm:bounded} has
the following corollary.

\begin{corollary}
  \label{cor:aps}
  Let $r \geq 3$ be a fixed integer and assume $p = p(n)\in (0,1)$ satisfies $p =
  o(n^{-1/(r-1)})$. Then, for every fixed $k\in\NN$, we have
  \[
    \Pr\big[\text{$[n]_p$ is $r$-AP-free}\big] = \exp\big(-\kappa_1 + \kappa_2 - \dotsb
    + (-1)^k \kappa_k +O(\Delta_{k+1})+o(1)\big)
  \]
  as $n\to\infty$.
  Moreover, if $p=o(n^{-1/(r-1)-\eps})$
  for some positive constant $\eps$, then there exists a positive integer
  $k=k(\eps,r)$ such that
  $\Delta_{k+1}=o(1)$.
\end{corollary}


In Section~\ref{j:sec:applications}, we will perform the necessary calculations
to determine the precise asymptotics of $\Pr[[n]_p\text{ is $r$-AP-free}]$ for $p=o(n^{-4/7})$.

\begin{corollary}\label{cor:3ap}
  If $p = o(n^{-4/7})$, then the probability that $[n]_p$ is $3$-AP-free is asymptotically
  \[
    \exp\Big(-\frac{n^2p^3}{4} + \frac{7n^3p^5}{24}\Big).
  \]
\end{corollary}

\subsection{Related work and open problems}

Janson's inequality was first proved (by Svante Janson himself) during the 1987
conference on random graphs in Pozna\'n, in response to Bollob\'as's
announcement of his estimate~\cite{Bo88} for the chromatic number of random
graphs, which requires a strong upper bound on the probability that a random
graph contains no large cliques. A related estimate was found, during the same
conference, by {\L}uczak. Janson's original proof was based on the analysis of
the moment-generating function of $X$, whereas {\L}uczak's proof used
martingales. Both of these arguments can be found in~\cite{JaLuRu90}. Our proof
of Theorem~\ref{j:thm:main} is inspired by a subsequent proof of Janson's
inequality that was found soon afterwards by Boppana and Spencer~\cite{BoSp89};
it uses only the Harris inequality. Somewhat later, Janson~\cite{Ja90} showed
that his proof actually gives bounds for the whole lower tail, and not just for
the probability $\Pr[X=0]$. Around the same time, Suen~\cite{Su90} proved a
correlation inequality that is very similar to Janson's. Suen's inequality
gives a slightly weaker estimate (which was later sharpened by
Janson~\cite{Ja98}), but is applicable in a much more general context. Another
generalisation of Janson's inequality was obtained recently by Riordan and
Warnke~\cite{RiWa15}.

In~\cite{Wo96}, Wormald proved that if $p=o(n^{-2/3})$, then
\begin{equation}
  \label{eq:k3gnp}
  \Pr[G_{n,p}\text{ is $K_3$-free}] = \exp\Big(-\frac{n^3p^3}{6}+\frac{n^4p^5}{4}-\frac{7n^5p^7}{12}+o(1)\Big),
\end{equation}
whereas for $G_{n,m}$ with $m=d\binom{n}{2}$ and $d=o(n^{-2/3})$, we have
\[
  \Pr[G_{n,m}\text{ is $K_3$-free}] = \exp\Big(- \frac{n^3d^3}{6} +o(1) \Big).
\]
These results were strengthened recently by Stark and Wormald~\cite{StWo}, who obtained the approximation in Corollary~\ref{cor:k3} (which
implies~\eqref{eq:k3gnp}) and also
\[
  \Pr[G_{n,m}\text{ is $K_3$-free}] = \exp\Big(- \frac{n^3d^3}{6} + \frac{n^2d^3}{2} -\frac{n^4d^6}{8} + o(1) \Big),
\]
where $m=d\binom{n}{2}$, which holds when $d=o(n^{-7/11})$. In fact, they were able to obtain a more general result,
which states that in the range where Corollary~\ref{cor:balanced-graphs} is applicable, the probability that $G_{n,p}$ or $G_{n,m}$
is $F$-free is approximated by the exponential of the first few terms of a power series in $n$ and $p$ (resp.\ $d$) whose terms depend only on $F$.
However, the way in which these terms are computed is rather implicit. In contrast, in the setting of binomial random subsets
such as $G_{n,p}$, our Theorem~\ref{j:thm:sparse} explains what these terms are.

While our results (and our methods) apply only to binomial subsets (e.g., $G_{n,p}$ and not $G_{n,m}$), the results for $G_{n,p}$
could conceivably be transferred to $G_{n,m}$ using the identity
\[
  \Pr[\text{$G_{n,m}$ is $F$-free}] =
  \frac{\Pr[G_{n,p}\text{ is $F$-free}] \cdot \Pr[e(G_{n,p}) = m \mid \text{$G_{n,p}$ is $F$-free} ]}{\Pr[e(G_{n,p})=m]}.
\]
It was shown by Stark and Wormald~\cite{StWo} that the conditional probability in the right-hand side
can be computed explicitly for a carefully chosen $p$ of the same order of magnitude as $d$.
However, this is not at all an easy task.

It would be interesting to establish a similar relationship in the more abstract
and general setting of random induced subhypergraphs. If this was possible, Theorem~\ref{j:thm:sparse}
could be used to count independent sets of a given (sufficiently small) cardinality in
general hypergraphs. In some sense, this would complement the counting results that
can be obtained with the so-called hypergraph container method developed by Balogh, Morris, and Samotij~\cite{BaMoSa15}
and by Saxton and Thomason~\cite{SaTh15}. Whereas the container method applies to somewhat large
independent sets, which exhibit a `global' structure, our Theorem~\ref{j:thm:sparse} would yield
estimates on the number of smaller independent sets that only exhibit `local' structure.
In particular, the container method can be used to estimate the probability that $G_{n,p}$ is $F$-free
whenever $p = \omega(n^{-1/m_2(F)})$ for every nonbipartite graph $F$. For $p$ in this range, $G_{n,p}$ conditioned on
being $F$-free is approximately $(\chi(F)-1)$-partite with very high probability. On the other hand, our method (and the method
of~\cite{StWo}) applies whenever $p=o(n^{-1/m_2(F)})$, provided that $F$ is $2$-balanced.
For $p$ in this range, the edges of $G_{n,p}$ conditioned on being $F$-free are still distributed very uniformly
with probability close to one.

\section{Proofs of Theorems~\ref{j:thm:sparse} and \ref{j:thm:bounded}}
\label{j:sec:discussion}

In this section, we will show that Theorem~\ref{j:thm:main} implies
Theorems~\ref{j:thm:sparse} and~\ref{j:thm:bounded}.
To prove Theorem~\ref{j:thm:sparse}, we need the following lemma, which also
clarifies the definition of $\Lambda_k$.

\begin{lemma}
  \label{lemma:lambdai}
  For every hypergraph $\Gamma=(\Omega,\cX)$, every $\p \in
  (0,1)^\Omega$,
  and every positive integer $k$, we have
  \[
    \Delta_{k+1}/\Delta_k \leq \Lambda_k(\Gamma,\p)
    \quad\text{and}\quad
    \delta_{k+1}/\delta_k \leq \Lambda_k(\Gamma,\p). 
  \]
\end{lemma}

\begin{proof}
  For every $V\in \mathcal C_{k+1}$ there exist at least two
  distinct $i \in V$ such that $V\setminus \{i\} \in \mathcal C_k$.
  Indeed, every connected graph with at least two vertices has at least two
  non-cut vertices. Therefore for each $V\in\mathcal C_{k+1}$ we can
  make a canonical choice of a set $V^-\subset V$ such that $V^-\in \mathcal
  C_k$ and
  \begin{equation}
    \label{eq:lambdai1}
    \max{\{\e[X_i]: i\in V\}}
    = \max{\{\e[X_i] : i \in V^-\}}.
  \end{equation}
  Denoting by $i_V$ the unique element in $V\setminus V^-$, we have
  $i_V\in \partial(V^-)$ because $G_\Gamma[V]$ is connected. Moreover,
  \[
    \Delta(\{X_i : i\in V\}) =
    \Delta(\{X_i : i\in V^-\}) \cdot \e[X_{i_V} \mid \prod_{i\in V^-}
    X_i = 1]
  \]
  and, analogously,
  \[
    \delta(\{X_i : i\in V\}) =
    \delta(\{X_i: i\in V^-\}) \cdot \e[X_{i_V} \mid \prod_{i\in V^-}
    X_i = 1]
  \]
  It follows that 
  \[
    \begin{split}
      \Delta_{k+1} & \leq \sum_{V^-\in \mathcal C_k} \Delta(\{X_i:
      i\in V^-\}) \sum_{j \in \partial(V^-)}
      \e[X_j \mid \prod_{i\in V^-}
      X_i = 1]\\ & = \sum_{V^-\in \mathcal C_k}
      \Delta(\{X_i : i\in V^-\})\cdot \lambda(V^-)
      \leq \Delta_k \cdot \Lambda_k(\Gamma,\p)
    \end{split} 
  \]
  and, similarly, $\delta_{k+1} \leq \delta_k \cdot \Lambda_k(\Gamma,\p)$.
\end{proof}

\begin{proof}[Proof of Thm.\ \ref{j:thm:sparse} from Thm.\ \ref{j:thm:main}]
  Assume that $\Gamma(n) = (\Omega(n),\cX(n))$ and
  $\p(n) = (p_\omega(n))_{\omega\in \Omega(n)}$
  are as in the statement of the theorem.

  Fix any $k\in \NN$ and $\eps\in (0,1)$ and let $K=K(k,\eps)$ be as given by
  Theorem~\ref{j:thm:main}. We verify that
  $\Gamma(n)$ and $\p(n)$ satisfy the assumption of
  Theorem~\ref{j:thm:main}
  for all sufficiently large $n$. For this, consider some nonempty $V \subseteq [N]$ of size at most $k+1$.
  Since $p=o(1)$, we have $\sum_{i\in
  V}\e[X_i] \leq (1-\eps)/2$
  for all sufficiently large $n$. Additionally, if $i\in \partial(V)$,
  then $\gamma_i$ intersects $\bigcup_{j\in V}\gamma_j$.
  Therefore,
  \[ 
    \sum_{i\in \partial(V)}
    \e[X_i]
    \leq \lambda(V) \cdot \max{\{p_\omega(n) :\omega \in \bigcup_{j\in V}\gamma_j\}}
    \leq(1-\eps)/2.
  \]
  By the union bound, this implies
  \[
    \rho_{k+1} = \max_{\substack{V\subseteq [N]\\1\leq |V|\leq k+1}}
    \Pr[X_i=1 \text{ for some $i\in V \cup
    \partial(V)$}] \leq 1-\eps.
  \]
  Therefore, Theorem~\ref{j:thm:main} yields
  \[
    |\log\Pr[X=0]+\kappa_1-\kappa_2+\dotsb+(-1)^{k+1}\kappa_k| \leq K \cdot
    (\delta_{1,K} + \Delta_{k+1,K}).
  \]
  Using
  Lemma~\ref{lemma:lambdai} and our assumption that $\Lambda_i(\Gamma(n),\p(n))=O(1)$
  for all constant $i$ (in particular, for all $1\leq i\leq K$), we get
  \[
    K \cdot \delta_{1,K} = K \cdot \sum_{i=1}^K\delta_i = O(\delta_1)
    \quad\text{and}\quad 
    K \cdot \Delta_{k+1,K} = K \cdot \sum_{i=k+1}^K\Delta_i =  O(\Delta_{k+1}),
  \]
  which completes the proof.
\end{proof}

\begin{lemma}
  \label{lemma:lambdaibounded}
  For all positive integers $k$ and $r$, there exist $k' = k'(k,r)$ and $K =
  K(k,r)$ such that, for every $p\in (0,1)$ and every hypergraph
  $\Gamma=(\Omega,\cX)$ with all edges of size at most $r$,
  \[
    \Delta_{k'}/\Delta_k \leq K\cdot \max{\{D(\Gamma,p), D(\Gamma,p)^{k'}\}}.
  \]
\end{lemma}
\begin{proof}
  Define $D^{(j)} = \max_{\emptyset \neq \Omega'\subseteq \Omega}
  d_j(\Omega')$ for every $j\geq 1$ and note that then $D(\Gamma,p) =
  \max_{j\geq 1} D^{(j)}p^j$.
  It is convenient to also define $D^{(0)} = 1$.

  We choose $k' = 2^{rk}$. Note that if $V\in \mathcal C_{k'}$, then there is an
  ordering of the elements of $V$ as $i_1,\dotsc,i_{k'}$ such that the
  set $\{i_1,\dotsc,i_\ell\}$ belongs to $\mathcal C_\ell$
  for all $\ell\in [k']$. 
  For every $\ell$, let
  $j_\ell = |\gamma_{i_\ell} \setminus (\gamma_{i_1}\cup \dotsb \cup\gamma_{i_{\ell-1}})|$. 
  Since $|\gamma_i|\leq r$ for all $i$, there are at most $2^{rk}-1$
  edges of $\Gamma$ that are completely contained in 
  $\gamma_{i_1}\cup \dotsb \cup\gamma_{i_k}$.
  Therefore, by our choice of $k'$, at least one of $j_{k+1},\dotsc,j_{k'}$ must be
  nonzero. Since there are at most
  $2^{r\ell}$ choices for the intersection of $\gamma_{i_\ell}$ and
  $\gamma_{i_1}\cup \dotsb \cup\gamma_{i_{\ell-1}}$, it then follows that
  \[ \Delta_{k'}/\Delta_k\leq \sum_{\substack{ 0\leq
  j_{k+1},\dotsc,j_{k'}\leq r\\ j_{k+1}+\dotsb + j_{k'}\geq 1}}
  \prod_{\ell=k+1}^{k'} 2^{r\ell}D^{(j_\ell)}p^{j_\ell}
  \leq K\cdot \max{\{ D(\Gamma,p), D(\Gamma,p)^{k'}\}} \]
  for an appropriate choice of $K$.
\end{proof}

\begin{proof}[Proof of Thm.\ \ref{j:thm:bounded} from Thm.\ \ref{j:thm:sparse}]
  Suppose that $\Gamma(n)=(\Omega(n),\cX(n))$ and $p(n)\in (0,1)$ are as in the
  statement of the theorem. Define
  the sequence $\p(n) = (p_\omega(n))_{\omega\in
  \Omega(n)}$ by $p_\omega(n) = p(n)$ for all $\omega\in \Omega(n)$.
  For every $V\subseteq [N]$,
  we have $|\bigcup_{i\in V}\gamma_i|\leq r|V|$, and so
  \[
    \begin{split}
      \lambda(V)
      &= \sum_{i\in \partial(V)} \e\big[X_i \mid
      \prod_{j\in V} X_j = 1\big]\\
      &\leq 
      2^{r|V|} + 
      \sum_{\emptyset \neq \Omega'\subseteq
      \bigcup_{i\in V}\gamma_i} \max_{j\geq 1} d_j(\Omega')
      p(n)^j\\
      &\leq 2^{r|V|}\big(1+D(\Gamma(n),p(n))\big).
    \end{split}
  \]
  Using our assumption on $D(\Gamma(n),p(n))$, this
  implies $\Lambda_k(\Gamma(n),\p(n))=O(1)$
  for every fixed $k\in \NN$. Since we also assume $p(n)\to 0$, Theorem~\ref{j:thm:sparse}
  implies the first statement of Theorem~\ref{j:thm:bounded}.

  To see the second statement, assume $D(\Gamma(n),p(n))\leq
  |\Omega(n)|^{-\eps}$ for a positive $\eps$.
  By Lemma~\ref{lemma:lambdaibounded}, iterated $2r/\eps$ times, we
  find that there are $k = k(\eps,r)$ and $K=K(\eps,r)$ such that
  $\Delta_{k}\leq K\cdot |\Omega(n)|^{-2r}\cdot \Delta_1$.
  Since $\Delta_1\leq
  |\Omega(n)|^rp(n)$, we obtain $\Delta_k \leq K\cdot |\Omega(n)|^{-r}
  \cdot p(n) = o(1)$.
\end{proof}

\section{Proof of Theorem \ref{j:thm:main}}\label{j:sec:proof}

Let $\Gamma$ and $\p$ be as in the statement of the theorem.
We start the proof by establishing some notational conventions.
Given a subset $V\subseteq [N]$, we use the abbreviations
\[ X_V = \prod_{i\in V} X_i \quad \text{and} \quad \overline X_V = \prod_{i\in
V} (1-X_i). \]
Note that these are the indicator variables for the events `$\gamma_i
\subseteq \Omega_\p$ for all $i\in V$' and `$\gamma_i\nsubseteq \Omega_\p$
for all $i\in V$', respectively. Besides being positively correlated
by Harris's inequality, the
variables $X_V$ satisfy the stronger FKG lattice condition
\begin{equation}
  \label{j:eq:lattice}
  \e[X_U]\e[X_V] \leq \e[X_{U\cup V}]\e[X_{U\cap V}] \quad \text{for all}\quad U,V \subseteq \cX.
\end{equation}
To see that this is true,
rewrite~\eqref{j:eq:lattice} using
 $\e[X_W] = \prod_{\omega\in \bigcup W} p_\omega$, take logarithms of
both sides, and note that
\[
  \begin{split}
    &\sum_{\omega\in \bigcup_{i\in U\cup V}\gamma_i} \log p_\omega\\
    ={} & \sum_{\omega\in \bigcup_{i\in U}\gamma_i} \log p_\omega +
    \sum_{\omega\in \bigcup_{i\in V}\gamma_i} \log p_\omega
    -
    \sum_{\omega\in (\bigcup_{i\in U}\gamma_i) \cap (\bigcup_{i\in V}\gamma_i)} \log p_\omega\\
    \geq{} &
    \sum_{\omega\in \bigcup_{i\in U}\gamma_i} \log p_\omega +
    \sum_{\omega\in \bigcup_{i\in V}\gamma_i} \log p_\omega 
    -
    \sum_{\omega\in \bigcup_{i\in U\cap V}\gamma_i} \log p_\omega,
  \end{split}
\]
since $\log p_\omega<0$ for all $\omega\in\Omega$ and $\bigcup_{i\in U\cap V}\gamma_i \subseteq
\bigcup_{i\in U}\gamma_i \cap \bigcup_{i\in V}\gamma_i$.

We will also use the notation
\[
  \mu_\pi = \prod_{P\in \pi} \e[X_P]
\]
whenever $\pi$ is a set of subsets of $[N]$ (usually a partition of some
subset of~$[N]$). Thus for a non-empty subset $V\subseteq [N]$, the value
\begin{equation}
  \label{eq:kappa'}
  \kappa(V) = \sum_{\pi\in \Pi(V)} (-1)^{|\pi|-1} (|\pi|-1)! \mu_\pi
\end{equation}
is the joint cumulant of $\{X_i : i\in V\}$. For the sake of brevity, we
will from now on write $\kappa(V)$ instead of the more cumbersome $\kappa(\{X_i: i\in V\})$.

Recall that we denote by $\partial(V)$ the external neighbourhood of
$V$ in the dependency graph, that is,
\[
  \partial(V) = N_{G_\Gamma}(V)\setminus V
\]
for every non-empty subset $V\subseteq [N]$.
We define
\begin{equation}
  \label{eq:rhoV}
  \rho_V = \Pr[\text{$X_i=1$ for some $i\in V \cup \partial(V)$}],
\end{equation}
so that $\rho_{k+1} = \max{\{\rho_V : V\subseteq [N]\text{ and } 1\leq |V|\leq k+1\}}$.
Moreover, we set
\[
  I(V) = [N]\setminus (V\cup \partial(V)).
\]
Neglecting the distinction between an index $i$ and the variable $X_i$, we
may say that $\partial(V)$ contains the variables outside of $V$ that are dependent on
$V$ and $I(V)$ contains those that are independent of $V$. 
As above, we write
$\mathcal C_i$ for the collection of all $i$-element sets $V\subseteq [N]$ such
that $G_\Gamma[V]$ is connected.
 We will also write $\mathcal
C_i(\ell)$ for the subset of $\mathcal C_i$ comprising all $A\in \mathcal C_i$ with $\max A = \ell$.

Assume that $k\in \NN$ and $\eps>0$ are
such that $\rho_{k+1} \leq 1-\eps$. 
Note that this implies, in particular, that
$\e[X_i]\leq 1-\eps$ for all $i\in [N]$.
Then we need to show that, for some $K = K(k,\eps)$,
\[
  \Big|\log \Pr[X=0] + \sum_{i\in [k]}
  (-1)^{i+1} \kappa_i\Big| \leq K \cdot (\delta_{1,K} + \Delta_{k+1,K}),
\]
where 
\[
  \delta_{1,K} = \sum_{i=1}^K \delta_i
  \quad\text{and}\quad
  \Delta_{k+1,K} = \sum_{i=k+1}^K \Delta_i.
\]
To do so, we first write out the probability that $X = 0$ using the chain rule:
\[
  \Pr[X=0] = \prod_{\ell\in [N]} \Pr[X_\ell = 0\mid \overline X_{[\ell-1]}=1] =
  \prod_{\ell\in [N]}\left(1-\e[X_\ell\mid \ol X_{[\ell-1]}=1]\right).
\]
Note that by the Harris inequality, $\e[X_\ell \mid \ol X_{[\ell-1]}=1] \le \e[X_\ell]
\leq 1-\eps$ .
Taking logarithms of both sides of the above equality and using the fact that $|\log (1-x) + x| \leq x^2/\eps$ for $x\in [0, 1-\eps]$, we get
\[
  \Big|\log\Pr[X=0] + \sum_{\ell\in [N]} \e[X_\ell\mid \ol X_{[\ell-1]} = 1]\Big|
  \leq \sum_{\ell\in [N]} \e[X_\ell\mid \ol X_{[\ell-1]} = 1]^2/\eps.
\]
Hence, using again
$\e[X_\ell \mid \ol X_{[\ell-1]}=1]\leq \e[X_\ell]$,
\begin{equation}
  \label{eq:step1}
  \Big|\log\Pr[X=0] + \sum_{\ell\in [N]} \e[X_\ell\mid \ol X_{[\ell-1]} = 1]\Big|
  \leq \sum_{\ell \in [N]} \e[X_\ell]^2 / \eps = \delta_1/\eps.
\end{equation}
Thus, our main goal becomes estimating the sum
\begin{equation}
  \label{eq:sum}
  \sum_{\ell\in[N]}\e[X_\ell\mid \ol X_{[\ell-1]}=1].
\end{equation}
We shall do this by approximating~\eqref{eq:sum} by an expression involving
the quantities
\begin{equation}
  \label{eq:defq}
  q(V,S) = \frac{(-1)^{|V|-1} \e[X_V]}{\e[\overline X_{S \setminus I(V)}\mid \overline X_{S\cap I(V)}=1]}.
\end{equation}
This ratio is well-defined for all $V, S \subseteq [N]$ because
\[
  \e[\ol X_{S \setminus I(V)} \mid \ol X_{S\cap I(V)}=1] \geq \e[\ol X_{S \setminus I(V)}] > 0,
\]
which is a consequence of the Harris inequality and the assumption that $p_\omega< 1$ for all $\omega\in \Omega$.
The relationship between~\eqref{eq:sum} and~\eqref{eq:defq} is made precise in the following lemma:

\begin{lemma}
  \label{lemma:top1}
  Let $k\in \NN$ and $\eps>0$ be such that $\rho_{k+1}\leq 1-\eps$. Then
  \[
    \Big|\sum_{\ell\in [N]} \e[X_\ell\mid \overline X_{[\ell-1]}=1] -
    \sum_{\ell\in [N]}\sum_{i\in [k]} \sum_{V\in \mathcal C_i(\ell)}
    q(V,[\ell-1])\Big| \leq 
    \Delta_{k+1}/\eps.
  \]
\end{lemma}

We postpone the proof of Lemma~\ref{lemma:top1} to Section~\ref{sec:proof-lemma-top1}
and instead show how it implies the assertion of the theorem. Before we can do this, we need
several additional definitions.

\begin{definition}[Attachment]
  \label{def:att}
  Given subsets $U,V\subseteq [N]$, we say that~$U$ \emph{attaches} to $V$,
  in symbols $U\att V$, if every connected component
  of $G_\Gamma[U\cup V]$ contains a vertex of $V$ (see Figure~\ref{j:fig:attach}).
\end{definition}

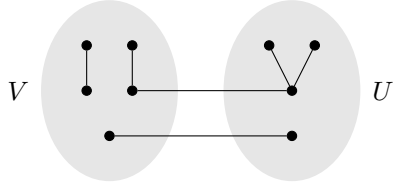
\begin{figure}
  \centering
  \begin{tikzpicture}[scale=1.2]
    \draw[fill=black!10,draw=none] (0,0.5) ellipse (0.75cm and 1cm);
    \draw[fill=black!10,draw=none] (2,0.5) ellipse (0.75cm and 1cm);
    \node (a) at (0.25,0.5) {};
    \node (x) at (-0.25,0.5) {};
    \node (b) at (0,0) {};
    \node (c) at (-0.25,1) {};
    \node (d) at (0.25,1) {};
    \draw (a) -- (d);
    \draw (x) -- (c);
    \node (a1) at (2,0) {};
    \node (b1) at (2,0.5) {};
    \node (c1) at (2-0.25,1) {};
    \node (d1) at (2.25,1) {};
    \draw (d1)--(b1)--(c1);
    \draw (b1)--(a);
    \draw (a1)--(b);
    \node[inner sep=0pt,fill=none,draw=none] at (-1,0.5) {$V$};
    \node[inner sep=0pt,fill=none,draw=none] at (3,0.5) {$U$};
  \end{tikzpicture}
  \caption{The set $U$ attaches to $V$, i.e., $U\att V$, but not vice-versa.}
  \label{j:fig:attach}
\end{figure}

We state the following simple facts for future reference:
\begin{enumerate}[label=(\roman*)]
\item
  We have $\emptyset \att V$ for every $V\subseteq [N]$.
\item
  If $i\in \partial(V)$, then $\{i\}\att V$.
\item
  If $U\att V$ and $W\att V$ then also $U\cup W\att V$.
\item
  If $V\in \mathcal C_{|V|}$ and $U \att V$, then $U\cup V\in \mathcal C_{|U \cup V|}$.
\end{enumerate}

\newcommand{\pic}{\Pi^{\mathsf C}}

\begin{definition}
  Suppose that $\emptyset \neq V\subseteq W\subseteq [N]$. We define
  \[
    \pic_V(W) \subseteq \Pi(W)
  \]
  to be the set of all partitions $\pi$ of $W$ that contain a part $P\in \pi$
  such that $V \subseteq P$ and $V$ is the union of connected components of
  $G_\Gamma[P]$ (see Figure \ref{j:fig:pic}).
\end{definition}

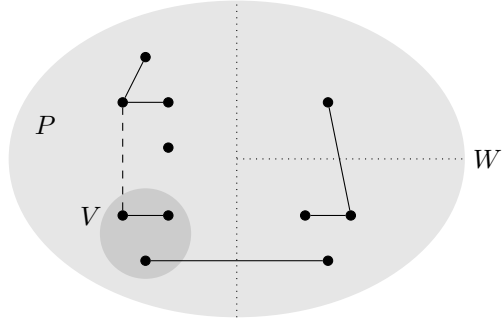
\begin{figure}
  \centering
  \begin{tikzpicture}[scale=1.2]
    \draw[fill=black!10,draw=none] (1,1.125) ellipse (2.5cm and 1.75cm);
    \draw[fill=black!20,draw=none] (0,0.3) circle (0.5cm);
    \node (1) at (0.25,0.5) {};
    \node (2) at (-0.25,0.5) {};
    \node (3) at (0,0) {};
    \node (5) at (0.25,1.25) {};
    \node (6) at (0.25,1.75) {};
    \node (7) at (-0.25,1.75) {};
    \node (8) at (0,2.25) {};
    \draw (1)--(2);
    \draw (6)--(7)--(8);
    \node (a) at (1.75,0.5) {};
    \node (b) at (2.25,0.5) {};
    \node (c) at (2,0) {};
    \node (d) at (2,1.75) {};

    \draw (3)--(c);
    \draw (a)--(b);
    \draw (d)--(b);
    \draw[dashed] (2) -- (7);

    \draw[dotted] (1,1.125-1.75) -- (1,1.125+1.75);
    \draw[dotted] (1,1.125) -- (1+2.5,1.125);
    \node[inner sep=0pt,fill=none,draw=none] at (-0.6,0.5) {$V$};
    \node[inner sep=0pt,fill=none,draw=none] at (-1.1,1.5) {$P$};
    \node[inner sep=0pt,fill=none,draw=none] at (3.75,1.125) {$W$};
  \end{tikzpicture}
  \caption{A partition in $\pic_V(W)$. Note that $V$ is the union of components of the subgraph induced by the part $P$ containing it. If the dashed edge were in $G_\Gamma$, then the partition would no longer be in $\pic_V(W)$.}
  \label{j:fig:pic}
\end{figure}

Next, for $\emptyset\neq V\subseteq W\subseteq [N]$, we define
\begin{equation}\label{eq:kappav}
  \kappa_V(W) = \sum_{\pi\in \pic_V(W)} (-1)^{|\pi|-1}(|\pi|-1)!\mu_\pi.
\end{equation}
Note that this is very similar to the definition \eqref{eq:kappa'} of $\kappa(W)$, 
except that we sum over $\pic_V(W)$ instead of $\Pi(W)$.
For every $k \in\NN$ and all $V,S\subseteq [N]$ with $V\neq \emptyset$, we set
\begin{equation}
  \label{eq:kappavk}
  \kappa^{(k)}_V(S) = \sum_{\substack{V\subseteq W\subseteq V\cup S\\ W \att V\\ |W| \leq k}}(-1)^{|W|-1}\kappa_V(W).
\end{equation}
Certainly, this is a very complicated definition, whose meaning
is far from clear at this point.
However, it serves as a convenient `bridge'
between $q(V,[\ell-1])$ and the values $\kappa_i$, as shown by the following two lemmas:

\begin{lemma}
  \label{lemma:top2}
  Let $k\in \NN$ and $\eps>0$ be such that $\rho_{k+1}\leq 1-\eps$. Then there is some
  $K=K(k,\eps)$ such that
  \[
    \Big| \sum_{\ell\in [N]}\sum_{i\in [k]} \sum_{V\in \mathcal C_i(\ell)} \big(
  q(V,[\ell-1]) - \kappa^{(k)}_V([\ell-1])\big)\Big| \leq  K \cdot (\delta_{1,K} + \Delta_{k+1,K}).
\]
\end{lemma}

\begin{lemma}\label{lemma:top3}
  For every $k\in \NN$, we have
  \[ 
    \sum_{\ell\in [N]}\sum_{i\in [k]}\sum_{V\in \mathcal C_i(\ell)}
    \kappa^{(k)}_{V}([\ell-1]) = \sum_{i\in [k]} (-1)^{i+1}\kappa_i.
  \]
\end{lemma}

Theorem~\ref{j:thm:main} is an easy consequence of
Lemmas~\ref{lemma:top1}, \ref{lemma:top2}, and~\ref{lemma:top3}. Indeed, 
assume $k\in \NN$ and $\eps>0$ are such that $\rho_{k+1}\leq 1-\eps$. It follows
from~\eqref{eq:step1}, the above three lemmas, and the triangle inequality
that
\[
  \Big|\log \Pr[X=0] + \sum_{i\in [k]} (-1)^{i+1} \kappa_i
  \Big| \leq \delta_1/\eps + \Delta_{k+1}/\eps + K' \cdot (\delta_{1,K'} + \Delta_{k+1,K'})
\]
for some $K' = K'(k,\eps)$. The assertion of the theorem now follows simply by observing that the right-hand side above is at most $K \cdot (\delta_{1,K} + \Delta_{k+1,K})$ for $K = K' + 1/\eps$.

\subsection{Proof of Lemma~\ref{lemma:top1}}
\label{sec:proof-lemma-top1}

We derive Lemma~\ref{lemma:top1} from the following auxiliary lemma,
which will also be used in the proof of Lemma~\ref{lemma:top2}.

\begin{lemma}
  \label{lemma:helper}
  Assume that $V, S \subseteq [N]$ are disjoint. Then for every 
  non-negative integer $k$,
  \begin{equation}
    \label{eq:helper}
    (-1)^k \cdot \e[X_V \mid \ol X_S=1] \leq (-1)^{k+|V|-1}
    \sum_{\substack{U\subseteq S, U\att V \\ |U|\leq k}} q(V\cup U, S).
  \end{equation}
\end{lemma}
\begin{proof}
  We claim that it suffices to prove that for every integer $k \ge 0$,
  \begin{equation}
    \label{eq:helper-alt}
    (-1)^k \cdot \e\big[X_V\ol X_S\big] \leq \sum_{\substack{U\subseteq S, 
    U\att V\\0\leq |U|\leq k}}(-1)^{k+|U|}
    \e\big[X_{V\cup U}\big]\e\left[\ol X_{S\cap I(V\cup U)}\right].
  \end{equation}
  Indeed, \eqref{eq:helper-alt} implies~\eqref{eq:helper} because
  \[
    \e\big[\ol X_{S\cap I(V\cup U)}\big] = {\Pr\big[\ol X_S =1\big]} \cdot
    {\e\big[\ol X_{S \setminus I(V\cup U)}\mid\ol X_{S\cap I(V\cup
    U)}=1\big]}^{-1}
  \]
  and because definition \eqref{eq:defq} gives
  \[
    q(V\cup U,S) = \frac{(-1)^{|V|+|U|-1} \e[X_{V\cup U}]}{\e[\overline X_{S \setminus I(V \cup U)} \mid \overline X_{S\cap I(V\cup U)}=1]}.
  \]

  We prove~\eqref{eq:helper-alt} by induction on $k$. When $k=0$, the inequality simplifies to
  \[
    \e[X_V \ol X_S] \leq \e[X_V] \e[\ol X_{S \cap I(V)}],
  \]
  which holds because $\ol X_S \leq \ol X_{S \cap I(V)}$ and because $X_V$ and
  $X_{S \cap I(V)}$ are independent. Assume now that $k\geq1$ and
  that~\eqref{eq:helper-alt} holds for all $k'$ with $0 \leq k' < k$. It
  follows from the Bonferroni inequalities that
  \begin{equation}
    \label{eq:Bonferroni}
    (-1)^k \cdot \ol X_{S \cap \partial(V)} \le (-1)^k \cdot \kern-0.5cm
    \sum_{\substack{U' \subseteq S \cap \partial(V) \\ |U'| \le k}}
    (-1)^{|U'|} X_{U'}.
  \end{equation}
  Since $S$ and $V$ are disjoint and $\partial(V) \cup I(V) = [N]\setminus V$, then multiplying~\eqref{eq:Bonferroni} through by $X_V \ol X_{S \cap I(V)}$ and taking expectations yields
  \begin{equation}
    \label{eq:Bonferroni-two}
    (-1)^k \cdot \e[X_V \ol X_S] \le \sum_{\substack{U' \subseteq S \cap \partial(V) \\ |U'| \le k}} (-1)^{k+|U'|} \e[X_{V \cup U'} \ol X_{S \cap I(V)}]
  \end{equation}
  Observe that for every $U' \subseteq S \cap \partial(V)$, the sets $V \cup U'$ and $S \cap I(V)$ are disjoint.
  In particular, if $U'$ is non-empty, then we may
  appeal to the induction hypothesis (with $k \leftarrow k-|U'|$) to bound each term
  in the right-hand side of~\eqref{eq:Bonferroni-two} as follows.
  As $S\cap I(V)\cap I(V \cup U'\cup U'') = S\cap I(V \cup U'\cup U'')$,
  \begin{multline}
    \label{eq:Bonferroni-substitution}
    (-1)^{k + |U'|} \cdot \e[X_{V\cup U'}\ol X_{S \cap I(V)}] \\
    \leq \sum_{\substack{U''\subseteq S\cap I(V)\\ U''\att V \cup U' \\0\leq |U''|\leq k-|U'|}} (-1)^{k+|U'|+|U''|}
    \e[X_{V \cup U'\cup U''}]\e[\ol X_{S\cap I(V \cup U'\cup U'')}].
  \end{multline}
  Finally, observe that every non-empty $U \subseteq S$ such that $U\att V$ can be
  partitioned into a non-empty $U' \subseteq S \cap \partial(V)$ and an $U'' \subseteq S \cap
  I(V)$ such that $U''\att (V \cup U')$ in a unique way. Indeed, one sets $U' = U \cap \partial(V)$ and $U'' = U \setminus U'$;
  this is the only such partition. Since $\emptyset \att V$ by definition, then bounding
  each term in~\eqref{eq:Bonferroni-two} that corresponds to a non-empty $U'$ using~\eqref{eq:Bonferroni-substitution}
  and rearranging the sum gives~\eqref{eq:helper-alt}.
\end{proof}

\begin{proof}[Proof of Lemma~\ref{lemma:top1}]
  Fix $\ell \in [N]$ and assume $k\in\NN$ and $\eps>0$ are such that $\rho_{k+1} \le 1- \eps$.
  Invoking Lemma~\ref{lemma:helper} with $V = \{\ell\}$ and $S = [\ell-1]$ twice,
  first with $k \leftarrow k-1$ and then with $k \leftarrow k$, to get both an upper
  and a lower bound on $\e[X_\ell \mid \ol X_{[\ell-1]}]$, we obtain
\begin{multline}
    \Big| \e[X_\ell \mid \ol X_{[\ell-1]}=1] - \sum_{\substack{U \subseteq
[\ell-1], U \att \{\ell\} \\ |U| \le k-1}} q(U \cup \{\ell\}, [\ell-1]) \Big|
\\
\le \Big| \sum_{\substack{U \subseteq [\ell-1], U \att \{\ell\} \\ |U| = k}}
q(U \cup \{\ell\}, [\ell-1]) \Big|.
\end{multline}
  Since the sets $U \cup \{\ell\}$ with $U \subseteq [\ell-1]$, $U \att \{\ell\}$, and $|U| = i-1$
  are precisely the elements of $\mathcal C_i(\ell)$, we can rewrite the above inequality as
  \begin{equation}
    \label{eq:X_ell-qVS}
    \Big| \e[X_\ell\mid \ol X_{[\ell-1]}=1] - \sum_{i\in[k]}\sum_{V\in \mathcal C_i(\ell)} q(V,[\ell-1]) \Big|
    \le \sum_{V\in \mathcal C_{k+1}(\ell)} | q(V,[\ell-1]) |.
  \end{equation}
  It follows from definition~\eqref{eq:defq} and Harris's inequality that
  \[
    \begin{split}
      |q(V,S)| & = \frac{\e[X_V]}{\e[\overline X_{S \setminus I(V)} \mid \overline X_{S\cap I(V)}=1]}\\
      &= \frac{\e[X_V]}{1 - \Pr[X_i = 1\text{ for some $i\in S \setminus I(V)$} \mid \ol X_{S \cap I(V)} = 1]} \le \frac{\e[X_V]}{1-\rho_V},
    \end{split}
  \]
  Since $\rho_V \le \rho_{k+1} \le 1-\eps$ for all $V$ with $|V| = k+1$,
  summing~\eqref{eq:X_ell-qVS} over all $\ell \in [N]$ yields
  \[
    \Big|\sum_{\ell\in [N]} \e[X_\ell\mid \overline X_{[\ell-1]}=1] - \sum_{\ell\in [N]}\sum_{i\in [k]} \sum_{V\in \mathcal C_i(\ell)} q(V,[\ell-1]) \Big|
    \leq \Delta_{k+1}/\eps, \]
  which is precisely the assertion of the lemma.
\end{proof}

\subsection{Proof of Lemma \ref{lemma:top2} -- preliminaries}

The goal of this subsection is to derive a recursive formula for
$\kappa_V(W)$, stated in Lemma~\ref{lemma:kappavrec} below, which will be used
in the proof of Lemma~\ref{lemma:top2}. 

\begin{definition}
  Suppose that $\emptyset \neq V\subseteq W\subseteq [N]$. We
  define $\Pi_V(W)$ and $\Pi^{\att}_V(W)$ as follows:
  \begin{enumerate}
    \item
      $\Pi_V(W)$ is the set of all partitions of $W$ that contain $V$ as a part;
    \item
      $\Pi^{\att}_V(W)$ is the set of all partitions $\pi\in \Pi_V(W)$ such
      that $P \att V$ for every part $P\in \pi$.
  \end{enumerate}
\end{definition}

Since by now we have defined several different classes of partitions of a set
$W$, it is a good moment to pause and convince ourselves that
\[
  \Pi^{\att}_V(W) \subseteq \Pi_V(W) \subseteq \pic_V(W) \subseteq \Pi(W).
\]
As a first step towards the promised recursive formula, we give an alternative expression for $\kappa_V(W)$.

\begin{definition}[Degree of a part in a partition]
  For a partition $\pi$ of a subset of $[N]$ and any part $P\in \pi$, let
  $d_\pi(P)$ denote the number of parts $P' \in \pi \setminus \{P\}$ such that
  $G_\Gamma$ contains an edge between $P'$ and $P$. We call $d_{\pi}(P)$ the
  \emph{degree} of $P$ in $\pi$.
\end{definition}

\begin{lemma}
  \label{lemma:kappav'}
  If $\emptyset\neq V\subseteq W\subseteq [N]$, then
  \[
    \kappa_V(W) = \sum_{\pi\in\Pi_V(W)} (-1)^{|\pi|-1} \chi_V(\pi) \mu_\pi,
  \]
  where
  \[
    \chi_V(\pi) =
    \begin{cases}
      1 &\text{if $|\pi|=1$}\\
      d_\pi(V)(|\pi|-2)! & \text{if $|\pi|\geq 2$}.
    \end{cases} 
  \]
\end{lemma}
\begin{proof}
  Given a $\pi\in \pic_V(W)$, let $P$ denote the part of $\pi$ containing $V$.
  Define a map $f \colon \pic_V(W) \to \Pi_V(W)$ as follows. If $P = V$, then let $f(\pi) = \pi$. Otherwise, let $f(\pi)$ be the partition obtained from $\pi$ by splitting
  $P$ into $V$ and $P \setminus V$. Clearly,
  \[ 
\begin{split}
    \kappa_V(W) & = \sum_{\pi\in \pic_V(W)} (-1)^{|\pi|-1}(|\pi|-1)!\mu_\pi\\
& =
    \sum_{\pi\in \Pi_V(W)} \sum_{\pi'\in f^{-1}(\pi)} (-1)^{|\pi'|-1}(|\pi'|-1)!\mu_{\pi'}.
\end{split}
  \]
  
  Observe that every $\pi\in \Pi_V(W)$ has exactly $|\pi|-d_\pi(V)$ preimages via $f$. One of them is $\pi$ itself
  and there are $|\pi|-1-d_\pi(V)$ additional partitions obtained from $\pi$ by merging $V$ with some other part $Q \in \pi$
  such that $G_\Gamma$ contains no edges between $V$ and $Q$. In particular, there is one preimage of size $|\pi|$ and there are
  $|\pi|- 1 - d_\pi(V)$ preimages of size $|\pi| - 1$. Furthermore, note that $\mu_{\pi'} = \mu_\pi$ for every $\pi' \in f^{-1}(\pi)$. Indeed, for every $Q \in \pi$ with no edges of $G_\Gamma$ between $Q$ and $V$, we have
  \[
    \e[X_V] \cdot \e[X_Q] = \e[X_V X_Q] = \e[X_{V \cup Q}].
  \]
  It follows that
  \[
    \begin{split}
      \kappa_V(W) & =\sum_{\pi\in \Pi_V(W)} (-1)^{|\pi|-1} \Big((|\pi|-1)! - (|\pi|-1-d_\pi(V))\cdot(|\pi|-2)!\Big)\mu_\pi \\
      & = \sum_{\pi\in\Pi_V(W)} (-1)^{|\pi|-1} \chi_V(\pi) \mu_\pi,
    \end{split} 
  \]
  as claimed.
\end{proof}

Our next lemma is the main result of this subsection and the
essential combinatorial ingredient of the proof of Lemma~\ref{lemma:top2}.
Stating it requires the following definition (illustrated in
Figure~\ref{j:fig:cut}).
\begin{definition}[$\cut_V(P)$]\label{def:cut}
  Suppose that $V\subseteq [N]$ is non-empty and $P\subseteq [N]$
  is disjoint from $V$ and satisfies $P\att V$. Then
  we write $\cut_V(P)$ for the
  collection of all sets $C\subseteq [N]$ satisfying $\partial(V)\cap P\subseteq C\subseteq P$ and $C\att V$.

\end{definition}

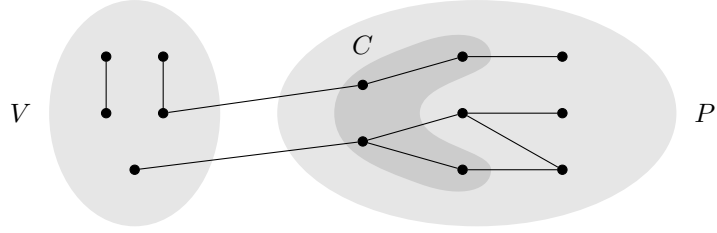
\begin{figure}
  \centering
  \begin{tikzpicture}[scale=1.5]
    \draw[fill=black!10,draw=none] (0,0.5) ellipse (0.75cm and 1cm);
    \draw[fill=black!10,draw=none] (3,0.5) ellipse (1.75cm and 1cm);
    \draw[draw=none,fill=black!20] plot [smooth cycle=1.2,tension=1] coordinates {(1.75,0.5) (2.5,1.125) (3.125,1) (2.5,0.5) (3.125,0) (2.5,-0.125)};
    \node (a) at (0.25,0.5) {};
    \node (x) at (-0.25,0.5) {};
    \node (b) at (0,0) {};
    \node (c) at (-0.25,1) {};
    \node (d) at (0.25,1) {};
    \draw (a) -- (d);
    \draw (x) -- (c);
    \node (a1) at (2,0.25) {};
    \node (b1) at (2,0.75) {};
    \draw (b1)--(a);
    \draw (a1)--(b);
    \node (a2) at (3-0.125,0) {};
    \node (b2) at (3-0.125,0.5) {};
    \node (c2) at (3-0.125,1) {};
    \node (a3) at (3.75,0) {};
    \node (b3) at (3.75,0.5) {};
    \node (c3) at (3.75,1) {};
    \draw (a1)--(a2)--(a3);
    \draw (a1)--(b2)--(a3);
    \draw (b2)--(b3);
    \draw (b1)--(c2)--(c3);
    \node[inner sep=0pt,fill=none,draw=none] at (-1,0.5) {$V$};
    \node[inner sep=0pt,fill=none,draw=none] at (5,0.5) {$P$};
    \node[inner sep=0pt,fill=none,draw=none] at (2,1.1) {$C$};
  \end{tikzpicture}
  \caption{A set $C$ in $\cut_V(P)$. Every element of $\cut_V(P)$, except
  for $P$ itelf, is a cutset in
  $G_\Gamma(V \cup P)$ that disconnects $V$ from some vertices in $P$.} \label{j:fig:cut}
\end{figure}

\begin{lemma}\label{lemma:kappavrec}
  Suppose that $\emptyset \neq V\subseteq W\subseteq [N]$ and $W \att V$.
  Then
  \begin{equation}\label{eq:kappavrec}
    \kappa_V(W) =
    \e[X_V] \sum_{\pi \in \Pi^{\att}_V(W)} (-1)^{|\pi|-1}(|\pi|-1)!
    \prod_{\substack{P\in \pi\\ P\neq V}} \sum_{C\in \cut_V(P)}\kappa_C(P).
  \end{equation}
\end{lemma}
\begin{proof}
  Denote the right hand side of~\eqref{eq:kappavrec} by $r_V(W)$. We need to show $\kappa_V(W) = r_V(W)$.
  Let us first rewrite the inner sum in~\eqref{eq:kappavrec}.  To this end, fix some non-empty
  $P \subseteq W\setminus V$ such that $P\att V$.  By the definition of $\kappa_C(P)$, see~\eqref{eq:kappav},
  \begin{equation}\label{eq:a1}
    \sum_{C\in \cut_V(P)}\kappa_C(P) = \sum_{C\in \cut_V(P)} \sum_{\pi\in
    \pic_C(P)} (-1)^{|\pi|-1}(|\pi|-1)! \mu_\pi.
  \end{equation}
  We may write this double sum more compactly as follows. For brevity, let $\partial_P(V) := \partial(V)\cap P$.
  Denote by $\tilde \Pi_V(P)$ the set of all partitions $\pi\in \Pi(P)$ such that some $Q \in \pi$ contains all neighbours of $V$
  in $P$, that is, such that $\partial_P(V) \subseteq Q$ for some $Q \in \pi$. We claim that
  \begin{equation}\label{eq:a2}
    \sum_{C\in \cut_V(P)} \kappa_C(P) = \sum_{\pi\in \tilde\Pi_V(P)}
    (-1)^{|\pi|-1}(|\pi|-1)! \mu_\pi.
  \end{equation}
  Indeed, this follows from \eqref{eq:a1} because, letting
  \[
    \mathcal Q(V,P) = \{(C,\pi) : C\in \cut_V(P) \text { and } \pi \in \pic_C(P)\},
  \]
  the projection $p_2\colon \mathcal Q(V,P) \ni (C, \pi) \mapsto \pi \in \Pi(P)$ is a bijection between $\mathcal Q(V,P)$ and $\tilde \Pi_V(P)$.
  This is because for every $(C, \pi) \in \mathcal Q(V,P)$, $C$ is the union of those connected components of $G_\Gamma(Q)$ that intersect $\partial_P(V)$.
  Furthermore, observe that the right-hand side of~\eqref{eq:a2} is simply the joint cumulant of the set
  \[
    P_V = \{X_i : i \in P\setminus \partial_P(V)\}\cup \{X_{\partial_P(V)}\},
  \]
  which is obtained from $P$ by replacing $\{X_i : i \in \partial_P(V)\}$ with the single variable $X_{\partial_P(V)}$.
  Therefore, it follows from~\eqref{eq:a2} that
  \begin{equation}\label{eq:a3}
    r_V(W) = \e[X_V]\sum_{\pi\in \Pi^{\att}_V(W)} (-1)^{|\pi|-1}(|\pi|-1)! \prod_{\substack{P\in \pi\\P\neq V}} \kappa(P_V).
  \end{equation}

  Let $\Pi'_V(W)$ be the set of all partitions in $\Pi_V(W)$ whose every part, except possibly $V$ itself,
  contains a neighbour of $V$. We claim that the product in the right-hand side of~\eqref{eq:a3} is zero
  for every $\pi \in \Pi_V'(W) \setminus \Pi_V^{\att}(W)$ and hence we may replace $\Pi_V^{\att}(W)$ with
  $\Pi_V'(W)$ in the range of summation in~\eqref{eq:a3}. Indeed, if $\pi \in \Pi_V'(W) \setminus \Pi_V^{\att}(W)$,
  then there is a $P \in \pi \setminus \{V\}$ such that $\partial_P(V) \neq \emptyset$ but $P \not\att V$.
  In particular, some connected component of $G_\Gamma[P]$ is disjoint from $\partial_P (V)$ and hence $\kappa(P_V) = 0$.  Expanding $\kappa(P_V)$ again, we obtain
  \begin{equation}\label{eq:a4}
    r_V(W) = \e[X_V]\kern-7pt\sum_{\pi\in \Pi'_V(W)}
    \kern-7pt(-1)^{|\pi|-1} (|\pi|-1)!
    \prod_{\substack{P\in \pi\\P\neq V}}
    \sum_{\pi'\in \tilde\Pi_V(P)}\kern-7pt (-1)^{|\pi'|-1}(|\pi'|-1)! \mu_{\pi'}.
  \end{equation}
  
  Let us write $\mathcal P$ to denote the set of all pairs $(\pi, \pi^*)\in \Pi'_V(W)\times \Pi_V(W)$ obtained as follows.
  Choose an arbitrary partition $\pi\in\Pi_V'(W)$ and refine every $P \in \pi\setminus\{V\}$ by replacing it by some
  $\pi_P \in \tilde\Pi_V(P)$, so that $\partial_P(V)$ is contained in a single part of $\pi_P$; finally, let $\pi^*$
  be the resulting partition of $W$.
  
  Suppose that $(\pi, \pi^*) \in \mathcal P$. Enumerate the parts of $\pi$ as
  $V, P_1, \dotsc, P_t$ and suppose that $\pi^*$ was obtained from $\pi$ by refining each $P_j$ into
  $i_j + 1$ parts, so that $|\pi^*| = t + 1 + i_1 + \dotsc + i_t$. Then, letting
  \[
    f(\pi,\pi^*) =  f_t(i_1,\dotsc,i_t) := (-1)^t t! \prod_{j\in [t]} (-1)^{i_j} i_j! = (-1)^{|\pi^*|-1} t! \prod_{j\in [t]}i_j!,
  \]
  we may rewrite~\eqref{eq:a4} as
  \begin{equation}\label{eq:terms}
    r_V(W) = \sum_{(\pi,\pi^*)\in \mathcal P} f(\pi,\pi^*)\mu_{\pi^*}.
  \end{equation}

  Fix some $\pi^*\in \Pi_V(W)$ and note that $\pi^*$ contains $d_{\pi^*}(V)$ parts other than $V$ that intersect $\partial(V)$.
  Let us write $s = |\pi^*|$,  $t = d_{\pi^*}(V)$, and
  $\pi^* = \{V, P_1^*, \dotsc, P_{s-1}^*\}$, so that $P_1^*, \dotsc, P_t^*$ are the parts intersecting $\partial(V)$.
  Fix an arbitrary permutation $\sigma$ of $[s-1]$ such that $\sigma(1) \in [t]$. Such a $\sigma$ can be used to define a $\pi$ such that
  $(\pi, \pi^*) \in \mathcal P$ in the following way. Consider the sequence $P_\sigma^* = (P_{\sigma(1)}^*, \dotsc, P_{\sigma(s-1)}^*)$.
  For every $i \in [t]$, let $P_i$ be the union of $P_i^*$ and all the $P_j^*$, with $j \in [s-1] \setminus [t]$, for which $P_i^*$ is the right-most element
  among $P_1^*, \dotsc, P_t^*$ that is to the left of $P_j^*$ in $P_\sigma^*$. (Since $\sigma(1) \in [t]$, then each $P_j^*$ with $j \in [s-1] \setminus [t]$
  has one of $P_1^*, \dotsc, P_t^*$ left of it.) A moment's thought reveals that each partition $\pi$ with $(\pi, \pi^*) \in \mathcal P$ is obtained this way
  from exactly $|f(\pi, \pi^*)|$ permutations $\sigma$. It follows that
  \[
    \begin{split}
      r_V(W) & = \sum_{\pi^* \in \Pi_V(W)} (-1)^{|\pi^*|-1} \mu_{\pi^*} \sum_{\substack{\pi \in \Pi'_V(W) \\ (\pi, \pi^*) \in \mathcal P}} |f(\pi, \pi^*)| \\
      & = \sum_{\pi^* \in \Pi_V(W)} (-1)^{|\pi^*|-1} \mu_{\pi^*} \cdot |\{ \sigma \in \Sym(|\pi^*|-1) : \sigma(1) \in \{1, \dotsc, d_{\pi^*}(V)\} \}| \\
      & = \sum_{\pi^* \in \Pi_V(W)} (-1)^{|\pi^*|-1} \mu_{\pi^*} \cdot \chi_V(\pi^*),
    \end{split}
  \]
  where $\chi_V(\pi^*)$ is as defined in Lemma~\ref{lemma:kappav'}. By Lemma~\ref{lemma:kappav'}, we conclude that $r_V(W) = \kappa_V(W)$, as required.
\end{proof}

\subsection{Proof of Lemma \ref{lemma:top2}}

For $V,S\subseteq [N]$ and $k\in \NN$ such that $|V| \leq k$, we define
\begin{equation}\label{eq:kappatilde}
  \tilde\kappa^{(k)}_V(S) = (-1)^{|V|-1} \e[X_V] \sum_{0\leq i\leq k-|V|}
  \Big(\sum_{\substack{U\subseteq S, U\att V\\ 1 \leq |U|\leq k-|V|}}
  \kappa^{(k-|V|)}_U(S\cap I(V))\Big)^i
\end{equation}
and
\begin{equation}\label{eq:qk}
  q^{(k)}(V,S) = (-1)^{|V|-1} \e[X_V]\sum_{0\leq i\leq k-|V|}
  \Big(\sum_{\substack{U\subseteq S, U\att V\\ 1 \leq |U|\leq k-|V|}} q(U,S\cap
  I(V))\Big)^i.
\end{equation}
Our proof of Lemma~\ref{lemma:top2} consists of three steps. First, in Lemma~\ref{lemma:qrec},
we show that $q(V, S) \approx q^{(k)}(V, S)$. Second, in Lemma~\ref{lemma:kappavkrec}, we show
that $\kappa_V^{(k)}(S) \approx \tilde\kappa^{(k)}_V(S)$. Finally, the fact that $q^{(k)}(V,S)$ and 
$\tilde\kappa_V^{(k)}(S)$ satisfy similar recurrences (given the above approximate equalities)
allows us to prove that also $q(V,S) \approx \kappa^{(k)}_V(S)$. Lemma~\ref{lemma:top2} then
follows easily. The precise definition of `$\approx$' above will be expressed by the following quantities.
For integers $k$ and $K$ satisfying $1\leq k\leq K$, define
\begin{equation}
  \Delta_k(V) = \sum_{\substack{U\att V\\ |U\cup V| = k\\}} \e[X_{U \cup V}]
  \qquad
  \text{and}
  \qquad
  \Delta_{k,K}(V) = \sum_{i=k}^K \Delta_i(V),
\end{equation}
and
\begin{equation}
  \delta_{k,K}(V) = \sum_{\substack{U\att V\\k\leq |U\cup V|\leq K\\}}
  \e[X_{U\cup V}]\max{\{\e[X_i] : i\in U\cup V\}}.
\end{equation}

\begin{lemma}\label{lemma:qrec}
  Let $\eps > 0$ and $k\in \NN$ be such that $\rho_k\leq 1-\eps$.
  Then there exists $K = K(k,\eps)$ such that for all
  $V,S\subseteq [N]$ with $1 \leq |V| \leq k$,
  \[
    |q(V,S) - q^{(k)}(V,S)| \leq K \cdot \big(\delta_{1,K}(V) + \Delta_{k+1,K}(V)\big).
  \]
\end{lemma}
\begin{proof}
  Fix $V$ and $S$ as in the statement of the lemma and set
  \[
    \rho = \Pr[X_i = 1\text{ for some $i\in S \setminus I(V)$} \mid \ol{X}_{S\cap I(V)} = 1].
  \]
  Then by definition
  \begin{equation} \label{eq:auxqVS}
    q(V,S) = \frac{(-1)^{|V|-1}\e[X_V]}{\e[\ol{X}_{S \setminus I(V)} \mid \ol{X}_{S\cap I(V)}=1]} = \frac{(-1)^{|V|-1}\e[X_V]}{1-\rho}.
  \end{equation}
  Since,
  by
  Harris's inequality
  and
  $|V|\leq k$, we have
  $0 \leq \rho \leq \rho_V
  \leq \rho_k \leq 1-\eps$,
  then~\eqref{eq:auxqVS}
  and the identity $(1-\rho)^{-1} = 1 + \rho + \dotsc + \rho^{k-|V|} + \rho^{k-|V|+1}(1-\rho)^{-1}$ yield
  \begin{equation}\label{eq:aVC-approx}
    \big|q(V,S) - (-1)^{|V|-1}\e[X_V]\cdot (1+\rho+\dotsb+\rho^{k-|V|})\big| \leq \eps^{-1}\e[X_V]\rho_V^{k-|V|+1}.
  \end{equation}
  We now observe that
  \[\begin{split}
    \e[X_V]\rho_V^{k-|V|+1} & \leq \e[X_V] \Big(\sum_{i\in V \cup \partial(V)}\e[X_i]\Big)^{k-|V|+1} \\
&= \e[X_V] \sum_{i_1, \dotsc, i_{k-|V|+1}} \prod_{j=1}^{k-|V|+1} \e[X_{i_j}]\end{split}
  \]
  and note that if $i_1, \dotsc, i_{k-|V|+1}$ are distinct elements of $\partial(V)$, then
  \[
    \e[X_V] \prod_{j=1}^{k-|V|+1} \e[X_{i_j}] \le \e[X_{V \cup \{i_1, \dotsc, i_{k-|V|+1}\}}]
  \]
  by Harris's inequality; if, on the other hand, either $i_j \in V$ for some $j$ or some two $i_j$ are equal, then Harris's inequality and the fact that $|\e[X_i]| \le 1$ for each $i$ imply the stronger bound
  \[
    \begin{split}
      &\e[X_V] \prod_{j=1}^{k-|V|+1} \e[X_{i_j}]\\
      \le{} & \e[X_{V \cup \{i_1, \dotsc, i_{k-|V|+1}\}}] \cdot \max\{\e[X_i] : i \in V \cup \{i_1, \dotsc, i_{k-|V|+1}\}\}.
    \end{split}
  \]
  In particular, the right-hand side of~\eqref{eq:aVC-approx} is bounded from above by
  \[
    \eps^{-1} \cdot (k-|V|+1)! \cdot \Delta_{k+1}(V) + \eps^{-1} \cdot k^{k-|V|+1} \cdot \delta_{1, k}(V),
  \]
  which yields
  \begin{equation}\label{eq:qrec1}
    \big|q(V,S) - (-1)^{|V|-1}\e[X_V]\cdot (1+\rho+\dotsb+\rho^{k-|V|})\big| \leq K_1 \cdot \big(\Delta_{k+1}(V) + \delta_{1,k}(V)\big)
  \end{equation}
  for some constant $K_1$ that depends only on $k$ and $\eps$.

  We claim that there is a constant $K_2 = K_2(k, \eps)$ such that, for 
  all $0 \leq i \leq k-|V|$,
  \begin{equation}\label{eq:qrec2}
    \e[X_V]\cdot
    \Big|\rho^i - \Big(\sum_{\substack{U\subseteq S, U\att V\\ 1 \leq |U|\leq k-|V|}}
    q(U,S\cap I(V))\Big)^i\Big| \leq K_2 \cdot \big( \delta_{1,K_2}(V) + \Delta_{k+1,K_2}(V) \big).
  \end{equation}
  Observe that \eqref{eq:qrec1} and \eqref{eq:qrec2} imply that
  \[
    | q(V,S) - q^{(k)}(V,S) | \leq K \cdot \big( \delta_{1,K}(V) + \Delta_{k+1,K}(V) \big)
  \]
  for some $K = K(k,\eps)$, giving the assertion of the lemma. It thus remains to prove \eqref{eq:qrec2}.

  We first consider the case $i = 1$. By the Bonferroni inequalities, for every positive $j$,
  \[
    (-1)^{j-1} \cdot \rho \le (-1)^{j-1} \cdot \sum_{\substack{U'\subseteq S \setminus I(V) \\1\leq |U'|\leq j}} (-1)^{|U'|-1}\e[X_{U'}\mid \ol{X}_{S\cap I(V)}=1].
  \]
  Applying Lemma~\ref{lemma:helper} with $k \leftarrow j-|U'|$, $V \leftarrow U'$, and $S \leftarrow S \cap I(V)$,
  we get that for each $U' \subseteq S \setminus I(V)$ with $1\leq |U'| \leq j$,
  \[
    (-1)^{j-|U'|}\e[X_{U'} \mid \ol{X}_{S\cap I(V)}=1] \leq
    \kern-10pt\sum_{\substack{U''\subseteq S\cap I(V), U''\att U'\\ |U''|\leq j-|U'|}}\kern-10pt
    (-1)^{j-1}q(U'\cup U'', S\cap I(V)).
  \]
  Next, observe that any non-empty $U\subseteq S$ with $U\att V$
  of size at most $j$ can be written uniquely as the disjoint union of $U'$ and $U''$, where $U' \subseteq V \cup \partial(V)$
  and $U'' \subseteq I(V)$ and $U''\att U'$. The previous two inequalities then imply that
  \begin{equation}
    \label{eq:lower-upper-x}
    (-1)^{j-1} \cdot \rho \leq (-1)^{j-1} \cdot \sum_{\substack{U\subseteq S, U\att V\\1\leq|U|\leq j}} q(U,S\cap I(V)).
  \end{equation}
  Invoking~\eqref{eq:lower-upper-x} twice, first with $j \leftarrow k-|V|$ and then with $j \leftarrow k-|V|+1$,
  to get both an upper and a lower bound on $\rho$, we obtain
  \begin{equation}\label{eq:qrec3}
    \begin{split}
      \Big|\rho - \sum_{\substack{U\subseteq S, U\att V\\ 1 \leq |U|\leq k-|V|}} q(U,S\cap I(V))\Big| & \leq
      \Big| \sum_{\substack{U\subseteq S, U\att V\\ |U| = k-|V|+1}} q(U,S\cap I(V)) \Big| \\
      & \le \sum_{\substack{U\subseteq S, U\att V\\ |U| = k-|V|+1}} \eps^{-1}\e[X_U],
    \end{split}
  \end{equation}
  where the last inequality uses the definition of $q(U, S\cap I(V))$ and the assumption that $\rho_k \leq 1-\eps$, see
  the discussion below~\eqref{eq:auxqVS}.

  Finally, we show how to deduce~\eqref{eq:qrec2} from \eqref{eq:qrec3}. Let 
  \[
    y = \sum_{\substack{U\subseteq S,U \att V\\1\leq |U|\leq k-|V|}} q(U,S\cap I(V)),
  \]
  so that the left-hand side of~\eqref{eq:qrec2} is $\e[X_V] \cdot |\rho^i - y^i|$,  and observe that, as in~\eqref{eq:qrec3},
  \[
    |y| \leq z := \sum_{\substack{U \att V\\1\leq |U|\leq k-|V|}} \eps^{-1}\e[X_U].
  \]
  Fix an $i \in \{1, \dotsc, k-|V|\}$. Since $|\rho| \le 1$, then
  \[
    |\rho^i - y^i| \leq |\rho-y| \cdot \sum_{j = 0}^{i-1} |\rho^jy^{i-1-j}| \le (1+z)^{i-1} \cdot |\rho - y|,
  \]
  which together with~\eqref{eq:qrec3} implies that
  \[
    \e[X_V]\cdot | \rho^i - y^i | \leq (1+z)^{i-1} \e[X_V] \sum_{\substack{U \att V\\|U|= k-|V|+1}}\eps^{-1}\e[X_U].
  \]
  Note that for pairwise disjoint $U_1, \dotsc, U_j \subseteq [N]$, Harris's inequality gives
  \[
    \prod_{\ell=1}^j \e[X_{U_\ell}] \le \e[ X_{U_1 \cup \dotsc \cup U_j}]
  \]
  and if $U_1, \dotsc, U_j \subseteq [N]$ are not pairwise disjoint, then the stronger FKG lattice condition~\eqref{j:eq:lattice} implies that
  \[
    \prod_{\ell=1}^j \e[X_{U_\ell}] \le \e[ X_{U_1 \cup \dotsc \cup U_j}] \cdot \max\{\e[X_i] : i \in U_1 \cup \dotsc \cup U_j\}.
  \]
  In particular, using a similar reasoning as
  used for deriving the bound~\eqref{eq:qrec1}
  from~\eqref{eq:aVC-approx}, we obtain
  \[
    (1+z)^{i-1} \e[X_V] \sum_{\substack{U \att V\\|U|= k-|V|+1}}\eps^{-1}\e[X_U] \le K_4 \cdot \big(\delta_{1, ik}(V) + \Delta_{k+1, ik+1}(V)\big)
  \]
  for sufficiently large $K_4=K_4(k,\eps)$. This shows \eqref{eq:qrec2} and
  hence the lemma.
\end{proof}

\begin{lemma}\label{lemma:kappavkrec}
  For every $k\in \NN$ there exists some $K=K(k)$ such that,
  for all 
   $V,S\subseteq [N]$ with $1 \leq |V| \leq k$, we have
  \[
    |\kappa^{(k)}_V(S) - \tilde\kappa^{(k)}_V(S)| \leq K \cdot \big( \delta_{1,K}(V) + \Delta_{k+1,K}(V) \big).
  \]
\end{lemma}
\begin{proof}
  Fix $k$, $S$, and $V$ as in the statement of the lemma and let
  \[
    x = \sum_{\substack{U\subseteq S, U\att V\\ 1 \leq |U|\leq k-|V|}} \kappa^{(k-|V|)}_U(S\cap I(V)),
  \]
  so that
  \begin{equation}
    \label{eq:tilde-kappa-x}
    \tilde \kappa_V^{(k)}(S) = (-1)^{|V|-1} \e[X_V](1+x+x^2+\dotsb +x^{k-|V|}).
  \end{equation}
  Using the definition~\eqref{eq:kappavk}, we may rewrite
  \begin{equation}
    \label{eq:x-def-expanded}
    x = \sum_{\substack{U\subseteq S, U\att V\\ 1 \leq |U|\leq k-|V|}} \sum_{\substack{U \subseteq W \subseteq U \cup (S \cap I(V)) \\ W \att U, |W| \le k - |V|}} (-1)^{|W|-1} \kappa_U(W).
  \end{equation}
  Recalling from Definition~\ref{def:cut} that
  \[
    \cut_V(W) = \{ U \subseteq W : U \att V \text{ and } \partial(V) \cap W \subseteq U\},
  \]
  we may switch the order of summation in~\eqref{eq:x-def-expanded} to obtain
  \[
    x = \sum_{\substack{W \subseteq S, W \att V\\ 1 \leq |W|\leq k-|V|}} \sum_{U\in \cut_V(W)}(-1)^{|W|-1} \kappa_U(W).
  \]
  For the sake of brevity, write
  \[
    f(W) = \sum_{U \in \cut_V(W)}(-1)^{|W|-1} \kappa_U(W).
  \]
  We may now rewrite~\eqref{eq:tilde-kappa-x} as
  \begin{equation}
    \label{eq:tilde-kappa-expanded}
    \tilde \kappa_V^{(k)}(S) = (-1)^{|V|-1} \e[X_V] \sum_{i = 0}^{k-|V|} \sum_{\substack{W_1, \dotsc, W_i \subseteq S \\ W_1, \dotsc, W_i \att V \\ 1 \leq |W_1|, \dotsc, |W_i| \leq k-|V| }} f(W_1) \cdot \dotsc \cdot f(W_i).
  \end{equation}
  
  Consider first the total contribution $\tilde \kappa_1$ to the right-hand side of~\eqref{eq:tilde-kappa-expanded} coming from terms corresponding to $W_1, \dotsc, W_i \subseteq S \setminus V$ that are pairwise disjoint and whose union has size at most $k - |V|$. Each such term may be regarded as a partition of the set $W = V \cup W_1 \cup \dotsc \cup W_i$, which satisfies $V \subseteq W \subseteq S$ and $|W| \leq k$; this partition $\{V, W_1, \dotsc, W_i\}$ belongs to $\Pi^{\att}_V(W)$. Conversely, given a $W$ with these properties, every partition $\pi \in \Pi^{\att}_V(W)$ corresponds to exactly $(|\pi|-1)!$ such terms; this is the number of ways to order the elements of $\pi \setminus \{V\}$ as $W_1, \dotsc, W_i$. Therefore,
  \[
    \tilde \kappa_1 = (-1)^{|V|-1}\e[X_V] \sum_{\substack{V\subseteq W \subseteq V \cup S \\ W \att V, |W|\leq k}}
    \sum_{\pi\in \Pi^{\att}_V(W)} (|\pi|-1)!\prod_{\substack{P\in \pi\\P\neq V}}f(P).
  \]
  In particular, Lemma~\ref{lemma:kappavrec} gives
  \[ 
    \tilde \kappa_1 = (-1)^{|V|-1} \sum_{\substack{V \subseteq W \subseteq V \cup S \\ W \att V, |W| \leq k}} (-1)^{|W|-|V|} \kappa_V(W) = \kappa^{(k)}_V(S).
  \]
  Every term in the right-hand side of~\eqref{eq:tilde-kappa-expanded} corresponding to $W_1, \dotsc, W_i$ that is not included in $\tilde \kappa_1$
  either satisfies $|V \cup W_1 \cup \dotsc \cup W_i| > k$ or the sets $V, W_1, \dotsc, W_i$ are not pairwise disjoint.
  Let $\tilde \kappa_2 = \tilde \kappa_V^{(k)}(S)-\tilde\kappa_1$ denote the total contribution of these terms. Since for every $W$, Harris's inequality implies
  \[
    |f(W)| \le \sum_{U \subseteq W} |\kappa_U(W)| \le \sum_{\pi \in \Pi(W)} |\pi|! \mu_\pi  \le |W|^{|W|} \e[X_W],
  \]
  there is a constant $K_1$ that depends only on $k$ such that
  \[
    |\tilde \kappa_2| \le K_1 \e[X_V] \sum_{W_1, \dotsc, W_i} \prod_{j=1}^i \e[X_{W_j}],
  \]
  where the sum ranges over all $i \leq k-|V|$ and $W_1, \dotsc, W_i \subseteq S$, each of size at most $k-|V|$ and attaching to $V$, such that either $|V \cup W_1 \cup \dotsc \cup W_i| > k$ or the sets $V, W_1, \dotsc, W_i$ are not pairwise disjoint. An argument analogous to the one given at the end of the proof of Lemma~\ref{lemma:qrec}, employing Harris's inequality and the stronger FKG lattice condition~\eqref{j:eq:lattice}, gives
  \[
    |\tilde \kappa_2| \le K \cdot \big( \delta_{1,K}(V) + \Delta_{k+1,K}(V) \big)
  \]
  for some $K$ that depends only on $k$.
\end{proof}

\begin{lemma}\label{lemma:bound}
  Let $k\in \NN$ be such that $\rho_k\leq 1-\eps$.
  Then there exists
  $K = K(k,\eps)$ such that for all $V,S\subseteq [N]$ with $1 \leq |V|\leq
  k$, we have
  \[
    |q(V,S) - \kappa^{(k)}_V(S)| \leq K \cdot \big( \delta_{1,K}(V)+\Delta_{k+1,K}(V) \big).
  \]
\end{lemma}
\begin{proof}
  We prove the lemma by complete induction on $k$. To this end, let $k\geq 0$ and suppose that
  the statement holds for all $k' \in \NN$ with $k' < k$. By the triangle inequality
  \[\begin{split}
    |q(V,S) - \kappa^{(k)}_V(S)| &\leq |q(V,S)-q^{(k)}(V,S)|\\
    &\qquad\qquad + |q^{(k)}(V,S) - \tilde \kappa^{(k)}_V(S)|\\
    &\qquad\qquad\qquad\qquad + |\tilde \kappa^{(k)}_V(S) - \kappa^{(k)}_V(S)|.
  \end{split}\]
  Lemmas~\ref{lemma:qrec} and~\ref{lemma:kappavkrec} imply that
  \[
    |q(V,S)-q^{(k)}(V,S)| + |\tilde \kappa^{(k)}_V(S) - \kappa^{(k)}_V(S)| \leq K_1 \cdot \big(\delta_{1,K_1}(V) + \Delta_{k+1,K_1}(V)\big)
  \]
  for some sufficiently large $K_1 = K_1(k,\eps)$ and thus it suffices to show that there is some $K_2= K_2(k,\eps)$ such that
  \begin{equation}\label{eq:z}
    |q^{(k)}(V,S) - \tilde \kappa^{(k)}_V(S)| \leq K_2 \cdot \big(\delta_{1,K_2}(V) + \Delta_{k+1,K_2}(V)\big).
  \end{equation}

  To this end, observe first that since $k-|V|<k$, then the induction hypothesis states that there is a constant $K' = K'(k, \eps)$ such that
  \begin{equation}\label{eq:z1}
    \big|q(U,S\cap I(V)) - \kappa_U^{(k-|V|)}(S\cap I(V))\big| \leq K' \cdot \big(\delta_{1,K'}(U) + \Delta_{k-|V|+1,K'}(U)\big)
  \end{equation}
  for all $U$ such that $1\leq|U|\leq k-|V|$. Let
  \[
    x = \sum_{\substack{U\subseteq S, U\att V\\ 1 \leq |U|\leq k-|V|}} \kappa^{(k-|V|)}_U(S\cap I(V))
  \]
  and, as in the proof of Lemma~\ref{lemma:qrec},
  \[
    y = \sum_{\substack{U\subseteq S, U\att V\\ 1 \leq |U|\leq k-|V|}} q(U,S\cap I(V)).
  \]
  Observe that
  \[
    |y| \le z:= \sum_{\substack{U \att V\\1\leq |U|\leq k-|V|}} \eps^{-1}\e[X_U],
  \]
  as in the proof of Lemma~\ref{lemma:qrec}, and that~\eqref{eq:z1} implies that
  \begin{equation}\label{eq:z2}
    |x-y| \leq w: = K' \cdot \sum_{\substack{U\att V\\1\leq |U|\leq k-|V|}} \big(\delta_{1,K'}(U) + \Delta_{k-|V|+1,K'}(U)\big).
  \end{equation}
  For any $i \ge 1$, we have
  \[
    |x^i - y^i| \le |x-y| \cdot \sum_{j=0}^{i-1} |x^j y^{i-1-j}| \le |x-y| \cdot (|x|+|y|)^{i-1} \le w (2z+w)^{i-1}.
  \]
  It follows that
  \begin{equation}
    \label{eq:z3}
    |q^{(k)}(V,S) - \tilde \kappa^{(k)}_V(S)| \le \sum_{1 \leq i \leq k-|V|} \e[X_V] \cdot w (2z+w)^{i-1}.
  \end{equation}
  Similarly as in the proofs of Lemmas~\ref{lemma:qrec} and~\ref{lemma:kappavkrec}, one sees that the FKG lattice condition~\eqref{j:eq:lattice}
  implies that the right hand side of~\eqref{eq:z3} is bounded from above by $K_2 \cdot \big( \delta_{1,K_2}(V) + \Delta_{k+1,K_2}(V) \big)$,
  provided $K_2=K_2(k, \eps)$ is sufficiently large, as claimed.
\end{proof}

\begin{proof}[Proof of Lemma \ref{lemma:top2}]
  It follows from Lemma~\ref{lemma:bound} that there is $K_1=K_1(k,\eps)$ such that
  \begin{multline}
    \Big| \sum_{\ell\in [N]}\sum_{i\in [k]} \sum_{V\in \mathcal C_i(\ell)}
    \big(q(V,[\ell-1]) - \kappa^{(k)}_V(S)\big)\Big|\\
    \leq
    \sum_{\ell\in [N]}\sum_{i\in [k]}
    \sum_{V\in \mathcal C_i(\ell)}
    K_1 \cdot \big(\delta_{1,K_1}(V)+\Delta_{k+1,K_1}(V)\big).\
  \end{multline}
  But if we choose $K$ sufficiently large then the right-hand side is
  at most $K \cdot \big(\delta_{1,K} + \Delta_{k+1,K}\big)$, as required.
\end{proof}

\subsection{Proof of Lemma \ref{lemma:top3}}

Fix an integer $k$ and an $\ell \in [N]$. Recalling~\eqref{eq:kappavk}, we rewrite the $\ell$th term of the sum
from the statement of the lemma as follows:
\[
  \sum_{i\in [k]}\sum_{V\in \mathcal C_i(\ell)} \kappa_V^{(k)}([\ell-1])
  = \sum_{i\in [k]} \sum_{V\in \mathcal C_i(\ell)} \sum_{\substack{V\subseteq W\subseteq V\cup [\ell-1]\\ W \att V\\|W| \leq k}} (-1)^{|W|-1} \kappa_V(W).
\]
It follows from Definition~\ref{def:att} that if $V$ is connected then $W \att V$ if and only if $W$ is connected. Therefore,
changing the order of the last two sums in the right-hand side of the above identity yields
\begin{equation}
  \label{eq:large-sum-kappa-W}
  \sum_{i\in [k]} \sum_{V\in \mathcal C_i(\ell)} \kappa_V^{(k)}([\ell-1])
  = \sum_{i\in [k]}\sum_{W\in \mathcal C_i(\ell)} \sum_{V\in \mathcal C_W} (-1)^{|W|-1}\kappa_V(W),
\end{equation}
where $\mathcal C_W$ denotes the collection of all connected sets $V \subseteq W$ satisfying $\max V = \max W$.

We claim that for each $W \in \mathcal C_i(\ell)$,
\begin{equation}
  \label{eq:kappa-W-kappa_V-W}
  \kappa(W) = \sum_{V\in \mathcal C_W}\kappa_V(W).
\end{equation}
Observe first that establishing this claim completes the proof of the lemma. Indeed, substituting~\eqref{eq:kappa-W-kappa_V-W} into~\eqref{eq:large-sum-kappa-W} and summing over all $\ell$ gives
\[
  \begin{split}
  \sum_{\ell \in [N]} \sum_{i\in [k]} \sum_{V\in \mathcal C_i(\ell)} \kappa_V^{(k)}([\ell-1]) & = \sum_{i\in [k]} \sum_{\ell \in [N]} \sum_{W\in \mathcal C_i(\ell)} (-1)^{|W|-1}\kappa(W) \\
  & = \sum_{i \in [k]} (-1)^{i-1} \sum_{W \in \mathcal C_i} \kappa(W) = \sum_{i \in [k]} (-1)^{i-1} \kappa_i.
  \end{split}
\]

Therefore, we only need to prove the claim. To this end, fix a $W \in \mathcal C_i(\ell)$. Recalling~\eqref{eq:kappa'} and~\eqref{eq:kappav}, it clearly suffices
to show that $\{\pic_V(W) : V \in \mathcal C_W\}$ is a partition of $\Pi(W)$. Obviously, $\pic_V(W) \subseteq \Pi(W)$ for each $V \in \mathcal C_W$. Conversely,
given an arbitrary $\pi \in \Pi(W)$, let $P \in \pi$ be the part containing $\max W$ and let $V$ be the connected component of $\max W$ in $G_\Gamma[P]$. Clearly,
$V \in \mathcal C_W$ and $\pi \in \pic_V(W)$. Moreover, the connected component of $\max W$ in $G_\Gamma[P]$ is the only set $V$ with this property, and so
the sets $\pic_V(W)$ and $\pic_U(W)$ are disjoint for distinct $U, V \in \mathcal C_W$.

\section{Computations}
\label{j:sec:applications}

The goal of this section is to carry out the necessary
computations for proving Corollaries~\ref{cor:k3c4}, \ref{cor:k3},
and~\ref{cor:3ap}. 

\subsection{Corollaries~\ref{cor:k3c4} and \ref{cor:k3}}

Assume that $\mathcal F = \{F_1,\dotsc,F_t\}$ is a collection of pairwise non-isomorphic $r$-graphs without isolated vertices
and let the associated hypergraph $\Gamma=(\Omega,\cX)$ be defined as in
Section~\ref{ssec:graphs}.
To prove Corollaries~\ref{cor:k3c4} and \ref{cor:k3}, we need to compute the
quantities $\kappa_k$ for small values of
$k$. This can be done using the following general approach: We first enumerate all
`isomorphism types' of
clusters in $\mathcal C_k$. Then we compute the joint cumulant for each
isomorphism type. Finally we multiply each value with the size of the
respective isomorphism class. This is made more precise as follows.

\begin{definition}
  An \emph{$\mathcal F$-complex} is a non-empty set of subgraphs of $K_n$, each
  of which is isomorphic to a graph in $\mathcal F$. An $\mathcal F$-complex
  $B$ is \emph{irreducible} if it cannot be written as the union of two
  $\mathcal F$-complexes $B_1$ and $B_2$ where every graph in $B_1$ is
  edge-disjoint from every graph in $B_2$. The set of all irreducible
  $\mathcal F$-complexes of cardinality $k$ is denoted by $\mathcal
  C_k(\mathcal F)$. The \emph{underlying graph} $G_B$ of an $\mathcal F$-complex $B$ is
  the subgraph of $K_n$ formed by taking the union of (the edge sets of) the
  graphs in $B$.
\end{definition}

Note that there is a natural bijection $\phi$ between the sets $V \subseteq
[N]$
of size $k$ and the $\mathcal F$-complexes of size $k$: $\phi$ maps  $V=
\{i_1,\dotsc,i_k\}$ to the
$\mathcal F$-complex $B = \{G_1,\dotsc,G_k\}$, where $G_j$ is the subgraph of
$K_n$ spanned by the edges in $\gamma_{i_j}$ (recall that $\gamma_{i_j}$ is a set of edges in $K_n$
and that we assume that none of
the graphs in $\mathcal F$ have isolated vertices). Note also that $\phi|_{\mathcal C_k}$ is a bijection
between $\mathcal C_k$ and $\mathcal C_k(\mathcal F)$. We can therefore write
$\kappa(B)$ for the joint cumulant of $\{X_i : i\in \phi^{-1}(B)\}$
without ambiguity, obtaining
\[
  \kappa_k = \sum_{B\in \mathcal C_k(\mathcal F)} \kappa(B).
\]
Using \eqref{j:eq:kappa} we easily express $\kappa(B)$ in terms of $G_B$:
\begin{equation}\label{eq:kappac}
  \kappa(B) = \sum_{\pi\in \Pi(B)} (|\pi|-1)!(-1)^{|\pi|-1} \prod_{B' \in \pi} p^{e_{G_{B'}}}.
\end{equation}

\begin{definition}
  Let $B_1$ and $B_2$ be $\mathcal F$-complexes. A map
  $f\colon V(G_{B_1}) \to V(G_{B_2})$ is a \emph{homomorphism} from $B_1$ to
  $B_2$ if for every graph $H \in B_1$, the graph $f(H)$
  (with vertex set $f(V(H))$ and edge set $\{\{f(u),f(v)\}
  : \{u,v\}\in E(H)\}$)
  belongs to $B_2$.
  If $f$ is bijective and both $f$ and $f^{-1}$
  are homomorphisms, then $f$ is
  an \emph{isomorphism}. We denote by $\Aut(B)$ the group of automorphisms of $B$,
  that is of isomorphisms from $B$ to $B$.
\end{definition}

It is easy to see that $\kappa$ assigns equal values to isomorphic $\mathcal F$-complexes. The following simple lemma can then be used to compute the values $\kappa_k$. In the sequel, we will denote by $n^{\ul i}$ the falling factorial $n(n-1)\dotsb(n-i+1)$.

\begin{lemma}\label{lemma:iso}
  Let $\mathcal C_k(\mathcal F)/{\cong}$ be the set of isomorphism types of
  $\mathcal F$-complexes in $\mathcal C_k(\mathcal F)$.
  Then
  \[ \sum_{B\in \mathcal C_k(\mathcal F)} \kappa(B)
  = \sum_{[B]\in \mathcal C_k(\mathcal F)/{\cong}} \kappa(B)
  \cdot \frac{n^{\ul {v_{G_B}}}}{|\Aut(B)|}.\]
\end{lemma}
\begin{proof}
  For each isomorphism type $[B]$, there are $n^{\ul{v_{G_B}}}$ ways to place the
  vertices of $G_B$ into $K_n$; this way, every element of $\mathcal C_k(\mathcal F)$
  isomorphic to $B$ is counted once for every automorphism of $B$.
\end{proof}

\begin{proof}[Proof of Corollary \ref{cor:k3c4}]
  Suppose that $\mathcal F= \{K_3,C_4\}$ and that $p= o(n^{-4/5})$. Since both $K_3$ and $C_4$ are $2$-balanced and
  \[
    \min\big\{m_2(K_3), m_2(C_4)\big\} = \min\{2, 3/2\} \geq 5/4,
  \]
  we can apply Corollary~\ref{cor:balanced-graphs}, which states that the probability that $G_{n,p}$ is simultaneously $K_3$-free and $C_4$-free is
  \[
    \exp\big(-\kappa_1+\kappa_2-\kappa_3+O(\Delta_4)+o(1)\big).
  \]

  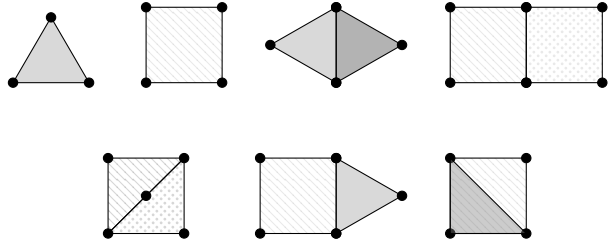
\begin{figure}
    \centering
    \begin{tikzpicture}[scale=1]
      \begin{scope}[shift={(-0.25,0)}]
        \mytriangle{(0,0)}{0}{fill=gray}
      \end{scope}
      \begin{scope}[shift={(1.5,0)}]
        \myquad{(0,0)}{0}{pattern=north west lines}
      \end{scope}
      \begin{scope}[shift={(4,0)}]
        \mytriangle{(0,0)}{90}{fill=gray}
        \mytriangle{(0,0)}{30}{fill=black}
      \end{scope}
      \begin{scope}[shift={(6.5,0)}]
        \myquad{(0,0)}{90}{pattern=north west lines}
        \myquad{(0,0)}{0}{pattern=crosshatch dots}
      \end{scope}
      \begin{scope}[shift={(1,-2)}]
        \coordinate (a) at (0,0);
        \coordinate (b) at (1,0);
        \coordinate (c) at (1,1);
        \coordinate (d) at (0,1);
        \coordinate (x) at (0.5,0.5) {};
        \draw[fill opacity=0.4,pattern=north west lines] (a) -- (x) -- (c) -- (d) -- cycle;
        \draw[fill opacity=0.4,pattern=crosshatch dots] (a) -- (x) -- (c) -- (b) -- cycle;
        \node at (a) {};
        \node at (b) {};
        \node at (c) {};
        \node at (d) {};
        \node at (x) {};
      \end{scope}
      \begin{scope}[shift={(4,-2)}]
        \myquad{(0,0)}{90}{pattern=north west lines}
        \mytriangle{(0,0)}{30}{fill=gray}
      \end{scope}
      \begin{scope}[shift={(5.5,-2)}]
        \myquad{(0,0)}{0}{pattern=north west lines}
        \draw[fill opacity=0.4,fill=gray] (0,0) -- (0,1) -- (1,0) -- cycle;
      \end{scope}
    \end{tikzpicture}
    \caption{The irreducible $\{K_3,C_4\}$-complexes of size at most two.
    Copies of $K_3$ are represented by the grey triangles 
    and copies of $C_4$ by the hatched or dotted $4$-cycles.}
    \label{fig:k3c4}
  \end{figure}

  Figure~\ref{fig:k3c4} shows all seven non-isomorphic irreducible
  $\mathcal F$-complexes of size at most two.
  Using Lemma~\ref{lemma:iso}, the contribution
  to $\kappa_k$ from a given $\mathcal F$-complex $B$ of size $k$ is
  \[ \kappa(B)\cdot\frac{n^{\ul{v_{G_B}}}}{|\Aut(B)|}. \]
  For the complexes shown in Figure~\ref{fig:k3c4}, we can easily calculate
  $|\Aut(B)|$ manually; going through the figure
  from the top left to the bottom right, we obtain the values
  \[ 6, 8, 4, 4, 4, 2, 2. \] Therefore
  \[ \kappa_1 = \frac{n^{\ul 3}p^3}{6} + \frac{n^{\ul 4}p^4}{8} \]
  and, since $p=o(n^{-4/5})$,
  \begin{align*}
    \kappa_2 & = \frac{n^{\ul 4}(p^5-p^6)}{4} + \frac{n^{\ul 6}(p^7-p^8)}{4} +
                 \frac{n^{\ul 5}(p^6-p^8)}{4}
             + \frac{n^{\ul 5}(p^6-p^7)}{2}
                 + \frac{n^{\ul 4}(p^5-p^7)}{2}\\
             & = \frac{n^{\ul 6}p^7}{4} + \frac{3n^{\ul 5}p^6}{4} +o(1).
  \end{align*}

  When calculating $\kappa_3$, we first observe that the underlying graphs of
  the third $\mathcal F$-complex and the fifth $\mathcal F$-complex in
  Figure~\ref{fig:k3c4} each contain a $C_4$ that is not already part of the
  complex and that the graph of the bottom right $\mathcal F$-complex contains
  a triangle that is not a part of the complex. Let $\kappa_3'$ denote the
  contribution of the two $\mathcal F$-complexes of size three that are
  obtained
  from one of these three complexes of size two by adding the `extra' $C_4$ or $K_3$. Then
  \[
    \kappa_3' = \frac{n^{\ul 4}(p^5-2p^8-p^9+2p^{10})}{4} + \frac{n^{\ul
    5}(p^6-3p^{10}+2p^{12})}{12} = \frac{n^{\ul 5}p^6}{12} + o(1).
  \]
  On the other hand, the contribution of every other $\mathcal F$-complex of to $\kappa_3$ is at most in the order
  of $(p + np^2 +n^2p^3)\cdot \kappa_2$, because, except in the two cases mentioned above, the graph of a complex
  of size three is obtained from the graph of a complex of size two by adding either a new edge, or a new vertex and
  two new edges, or two new vertices and three new edges. Using the assumption $p = o(n^{-4/5})$, we get 
  \[
    (p + np^2 +n^2p^3)\cdot \kappa_2 = O(n^6p^8 + n^5p^7 + n^7p^9 + n^8p^{10}) = o(1),
  \]
  and therefore
  \[
    \kappa_3 = \frac{n^{\ul 5}p^6}{12} +o(1).
  \]
  Since the $\mathcal F$-complexes accounted for by $\kappa_3'$
  are `complete' (in the sense that their graphs do not contain
  either a $K_3$ or a $C_4$ that is not already a part of the complex),
  a similar reasoning shows that
  \[
    \Delta_4 \leq O\big((p+np^2+n^2p^3) \cdot \kappa_3'\big) + O\big((1+p+np^2+n^2p^3) \cdot (\kappa_3 - \kappa_3')\big) = o(1).
  \]
  Since our assumption on $p$ implies that $\max\{\kappa_1, \kappa_2,
  \kappa_3\} = o(n)$, we can replace $n^{\ul i}$
  by $n^i$
  in the expressions for $\kappa_1,\kappa_2,\kappa_3$, incurring only an additive error of $o(1)$. Thus the probability that $G_{n,p}$ with $p = o(n^{-4/5})$
  is simultaneously  triangle-free and $C_4$-free is asymptotically
  \[
    \exp\Big(- \frac{n^3p^3}{6} - \frac{n^4p^4}{8} + \frac{n^6p^7}{4} + \frac{2n^5p^6}{3}\Big),
  \]
  as claimed.
\end{proof}

\begin{proof}[Proof of Corollary \ref{cor:k3}]
  Suppose that $\mathcal F= \{K_3\}$ and $p= o(n^{-7/11})$. Since $K_3$ is
  $2$-balanced and $m_2(K_3) = 2 \geq 11/7$,
  we can apply Corollary~\ref{cor:balanced-graphs}, which tells us that the probability that $G_{n,p}$ is triangle-free is
  \[
    \exp\big(-\kappa_1+\kappa_2-\kappa_3+\kappa_4+O(\Delta_5)+o(1)\big).
  \]

  In Figure~\ref{fig:k3} we see representations of all isomorphism types of irreducible $\mathcal F$-complexes of size up to four.
  Generating a similar list of complexes of size five would most likely require the help of a computer.

  \begin{figure}
    \centering
    \begin{tikzpicture}[scale=1]
      \begin{scope}[shift={(-0.5,0.5)}]
        \mytriangle{(0,0)}{0}{fill=gray}
      \end{scope}
      \begin{scope}[shift={(1.5,0.5)}]
        \mytriangle{(0,0)}{90}{fill=gray}
        \mytriangle{(0,0)}{30}{fill=black}
      \end{scope}
      \begin{scope}[shift={(3.5,0.5)}]
        \mytriangle{(0,0)}{120}{fill=gray}
        \mytriangle{(0,0)}{60}{fill=black}
        \mytriangle{(0,0)}{0}{pattern=north west lines}
      \end{scope}
      \begin{scope}[shift={(5.5,0.5)}]
        \coordinate (x) at (1,1);
        \coordinate (y) at (1,0);
        \mytriangle{(0,0)}{90}{fill=gray}
        \draw[fill opacity=0.4, fill=black] (a)--(x)--(b)--cycle;
        \draw[fill opacity=0.4, pattern=north west lines] (a)--(y)--(b)--cycle;
        \node at (x) {};
        \node at (y) {};
      \end{scope}
      \begin{scope}[shift={(8,1)}]
        \coordinate (a) at (0.86603,-0.5);
        \coordinate (b) at (-0.86603,-0.5);
        \coordinate (c) at (0,1);
        \coordinate (x) at (0,0);
        \draw[fill opacity=0.4, fill=gray] (a)--(x)--(b)--cycle;
        \draw[fill opacity=0.4, fill=black] (a)--(x)--(c)--cycle;
        \draw[fill opacity=0.4, pattern=north west lines] (b)--(x)--(c)--cycle;
        \node at (a) {};
        \node at (b) {};
        \node at (c) {};
        \node at (x) {};
      \end{scope}
      \begin{scope}[shift={(0,-2)}]
        \mytriangle{(0,0)}{0}{fill=gray}
        \mytriangle{(1,0)}{0}{fill=black}
        \mytriangle{(1,0)}{60}{pattern=north west lines}
        \mytriangle{(0,0)}{60}{pattern=crosshatch dots}
      \end{scope}
      \begin{scope}[shift={(3,-2)}]
        \mytriangle{(0,0)}{60}{fill=gray}
        \mytriangle{(0,0)}{0}{fill=black}
        \mytriangle{(1,0)}{60}{pattern=north west lines}
        \coordinate (x) at (1,1.73205);
        \draw[fill opacity=0.4,pattern=crosshatch dots] (x)--(c)--(b)--cycle;
        \node at (x) {};
      \end{scope}
      \begin{scope}[shift={(5,-2)}]
        \mytriangle{(0,0)}{0}{fill=gray}
        \mytriangle{(1,0)}{0}{fill=black}
        \mytriangle{(1,0)}{60}{pattern=north west lines}
        \mytriangle{(b)}{120}{pattern=crosshatch dots}
      \end{scope}
      \begin{scope}[shift={(8,-2)}]
        \mytriangle{(0,0)}{60}{fill=gray}
        \coordinate (x) at (-0.5,1.73205);
        \coordinate (y) at (0.5,1.73205);
        \draw[fill opacity=0.4,pattern=north west lines] (x)--(b)--(c)--cycle;
        \draw[fill opacity=0.4,pattern=crosshatch dots] (y)--(b)--(c)--cycle;
        \node at (x) {};
        \node at (y) {};
        \mytriangle{(0,0)}{0}{fill=black}
      \end{scope}
      \begin{scope}[shift={(0.75,-4)}]
        \coordinate (a) at (0,0);
        \coordinate (b) at (0,1);
        \coordinate (x1) at (-1,0);
        \coordinate (x2) at (1,0);
        \coordinate (y1) at (-1,1);
        \coordinate (y2) at (1,1);
        \draw[fill opacity=0.4,fill=gray] (a)--(x1)--(b)--cycle;
        \draw[fill opacity=0.4,pattern=north west lines] (a)--(x2)--(b)--cycle;
        \draw[fill opacity=0.4,fill=black] (a)--(y1)--(b)--cycle;
        \draw[fill opacity=0.4,pattern=crosshatch dots] (a)--(y2)--(b)--cycle;
        \node at (a) {};
        \node at (b) {};
        \node at (x1) {};
        \node at (x2) {};
        \node at (y1) {};
        \node at (y2) {};
      \end{scope}
      \begin{scope}[shift={(3.5,-3.5)}]
        \coordinate (a) at (0.86603,-0.5);
        \coordinate (b) at (-0.86603,-0.5);
        \coordinate (c) at (0,1);
        \coordinate (x) at (0,0);
        \draw[fill opacity=0.4, fill=gray] (a)--(x)--(b)--cycle;
        \draw[fill opacity=0.4, fill=black] (a)--(x)--(c)--cycle;
        \draw[fill opacity=0.4, pattern=north west lines] (b)--(x)--(c)--cycle;
        \draw[fill opacity=0.4, pattern=crosshatch dots] (a)--(b)--(c)--cycle;
        \node at (a) {};
        \node at (b) {};
        \node at (c) {};
        \node at (x) {};
      \end{scope}
      \begin{scope}[shift={(0.75,-5.5)}]
        \coordinate (a) at (0.86603,-0.5);
        \coordinate (b) at (-0.86603,-0.5);
        \coordinate (c) at (0,1);
        \coordinate (d) at (0.86603,0.5);
        \coordinate (x) at (0,0);
        \draw[fill opacity=0.4, fill=gray] (a)--(x)--(b)--cycle;
        \draw[fill opacity=0.4, fill=black] (a)--(x)--(c)--cycle;
        \draw[fill opacity=0.4, pattern=north west lines] (b)--(x)--(c)--cycle;
        \draw[fill opacity=0.4, pattern=crosshatch dots] (x)--(d)--(c)--cycle;
        \node at (a) {};
        \node at (b) {};
        \node at (c) {};
        \node at (d) {};
        \node at (x) {};
      \end{scope}
      \begin{scope}[shift={(5.15,-6)}]
         \coordinate (a) at (0,0.951);
         \coordinate (b) at (0.31,0);
         \coordinate (c) at (1.31,0);
         \coordinate (d) at (1.62,0.951);
         \coordinate (x) at (0.81,1.694);
         \draw[fill opacity=0.4, fill=gray] (a)--(b)--(c)--cycle;
         \draw[fill opacity=0.4, fill=black] (d)--(b)--(c)--cycle;
         \draw[fill opacity=0.4, pattern=crosshatch dots] (a)--(b)--(x)--cycle;
         \draw[fill opacity=0.4, pattern=north west lines] (c)--(d)--(x)--cycle;
         \node at (a) {};
         \node at (b) {};
         \node at (c) {};
         \node at (d) {};
         \node at (x) {};
      \end{scope}
      \begin{scope}[shift={(3.5,-5.5)}]
        \coordinate (a) at (0.86603,-0.5);
        \coordinate (b) at (-0.86603,-0.5);
        \coordinate (c) at (0,1);
        \coordinate (d) at (0.86603,0.5);
        \coordinate (x) at (0,0);
        \draw[fill opacity=0.4, fill=gray] (a)--(x)--(b)--cycle;
        \draw[fill opacity=0.4, fill=black] (a)--(x)--(c)--cycle;
        \draw[fill opacity=0.4, pattern=north west lines] (b)--(x)--(c)--cycle;
        \draw[fill opacity=0.4, pattern=crosshatch dots] (a)--(d)--(c)--cycle;
        \node at (a) {};
        \node at (b) {};
        \node at (c) {};
        \node at (d) {};
        \node at (x) {};
      \end{scope}
      \begin{scope}[shift={(7.5,-6)}]
        \coordinate (a) at (0,0);
        \coordinate (b) at (1,0);
        \coordinate (c) at (1,1);
        \coordinate (d) at (0,1);
        \coordinate (x) at (0.5,0.5);
        \draw[fill opacity=0.4,fill=gray] (a)--(x)--(b)--cycle;
        \draw[fill opacity=0.4,pattern=north west lines] (a)--(x)--(d)--cycle;
        \draw[fill opacity=0.4,fill=black] (b)--(x)--(c)--cycle;
        \draw[fill opacity=0.4,pattern=crosshatch dots] (d)--(x)--(c)--cycle;
        \node at (a) {};
        \node at (b) {};
        \node at (c) {};
        \node at (d) {};
        \node at (x) {};
      \end{scope}
   \end{tikzpicture}
    \caption{The irreducible $\{K_3\}$-complexes of size at most four. The four
    complexes in the bottom row are negligible when
    $p=o(n^{-7/11})$.}\label{fig:k3}
  \end{figure}
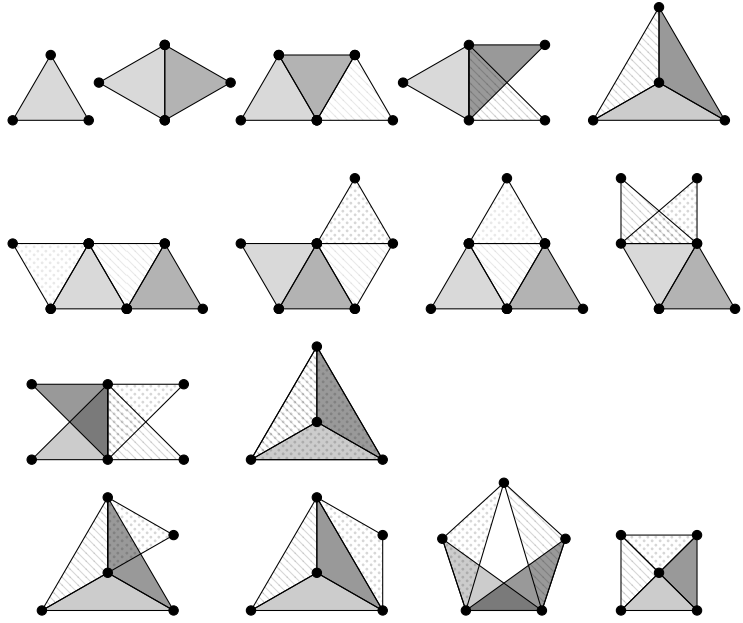

  By Lemma~\ref{lemma:iso}, the contribution to $\kappa_k$ from the isomorphism type of an $\mathcal F$-complex $B$ of size $k$ is
  \[
    \kappa(B)\cdot\frac{n^{\ul{v_{G_B}}}}{|\Aut(B)|}.
  \]
  For the complexes shown in Figure~\ref{fig:k3}, it is not too difficult to calculate $|\Aut(B)$| by hand. In fact, since the automorphism group
  of $K_3$ comprises all $3!$ permutations of $V(K_3)$, automorphisms of $\{K_3\}$-complexes are simply automorphisms of the $3$-uniform hypergraphs
  involved\footnote{But for general $\mathcal F$, it is wrong to think of an $\mathcal F$-complex isomorphism as a hypergraph isomorphism.}.
  For example, the leftmost $\mathcal F$-complex in the second row has exactly two
  automorphisms: the trivial one, and the unique automorphism exchanging the vertices belonging to exactly one triangle.
  Under our assumptions on $p$, we have $\kappa_k = \Delta_k + o(1)$ for $k \in \{3, 4\}$. This is the case because $|\kappa_k - \Delta_k| = O(p\Delta_k)$
  and
  \[
    p\Delta_3 \leq O(n^5p^8 + n^4p^7) = o(1) \qquad \text{and} \qquad p\Delta_4 \leq p \cdot O(1+p+p^2n) \cdot \Delta_3 = o(1),
  \]
  as can be seen by looking at Figure~\ref{fig:k3}.
  
  Now we just work through the figure row by row (from the top left to the bottom right) and in this order, we compute (using the first row)
  \begin{align*}
    &\kappa_1 = \frac{n^{\ul 3}p^3}{6},\\
    &\kappa_2 = \frac{n^{\ul 4}(p^5-p^6)}{4},\\
    &\kappa_3 = \Delta_3 + o(1) = \frac{n^{\ul 5}p^7}{2} + \frac{n^{\ul 5}p^7}{12} + \frac{n^{\ul
    4}p^6}{6} + o(1),\\
    \intertext{and (using the other rows)}
    &\kappa_4 = \Delta_4 + o(1) = \frac{n^{\ul 6}p^9}{2} + \frac{n^{\ul 6}p^9}{2}
  + \frac{n^{\ul 6}p^9}{6} + \frac{n^{\ul 6}p^9}{2}
  + \frac{n^{\ul 6}p^9}{48} + \frac{n^{\ul 4}p^6}{24}\\
    &\hspace{7.85cm}+ O(n^5p^8) + o(1).
  \end{align*}
  The term $O(n^5p^8)$ represents the contribution of the four complexes in the bottom row of Figure~\ref{fig:k3}, which is
  $o(1)$, as $p = o(n^{-7/11})$. Finally, we have
  \[
    \Delta_5 = O(p\Delta_4 + np^2\Delta_4 + n^5p^8 +n^5p^9) = O(n^4p^7 + n^5p^8 + n^6p^{10} + n^7p^{11}) = o(1),
  \]
  since the graph of an $\mathcal F$-complex of size five must be obtained by adding either a new edge or a new vertex and two new
  edges to one of the graphs in Figure~\ref{fig:k3}, or else it must be isomorphic to one of the first three graphs
  in the bottom row of Figure~\ref{fig:k3} (as the graphs of the remaining complexes of size four contain only triangles that are already in the complex).

  Finally, $\kappa_1 = n^{\ul 3}p^3/6 = (n^3-3n^2)p^3/6 + o(1)$ and, since $\max\{\kappa_2, \kappa_3, \kappa_4\} = o(n)$,
  we may replace the falling factorials $n^{\ul i}$ in the remaining expressions by $n^i$.
  Adding up the terms in $-\kappa_1 + \kappa_2 - \kappa_3 + \kappa_4$, we obtain that the
  probability that $G_{n,p}$ with $p = o(n^{-7/11})$ is triangle-free is asymptotically
  \[
    \exp\Big(- \frac{n^3p^3}{6} + \frac{n^4p^5}{4}
    - \frac{7n^5p^7}{12} 
    +\frac{n^2p^3}{2} - \frac{3n^4p^6}{8}
    + \frac{27n^6p^9}{16} \Big),
  \]
  as claimed.
\end{proof}

\subsection{Corollary~\ref{cor:3ap}} It only remains
to prove 
Corollary~\ref{cor:3ap}.

\begin{proof}[Proof of Corollary \ref{cor:3ap}]
  Let $\Gamma$ be the hypergraph of $r$-APs in $[n]$,
  as defined in Section~\ref{ssec:aps}, 
  and assume that $p =
  o(n^{-4/7})$. Then by Corollary~\ref{cor:aps} with $r=3$ and $k=2$,
  \[
    \Pr[X=0] = \exp\big(-\kappa_1+\kappa_2+O(\Delta_3)+o(1)\big).
  \]
  It remains to calculate $\kappa_1$, $\kappa_2$, and $\Delta_3$. For  $i\in[n]$, the number of $3$-APs containing $i$ is
  \[
    f(i) = \frac{n}{2} + \min{\{i,n-i\}}+O(1),
  \]
  where $\min{\{i,n-i\}}$ counts the $3$-APs that have $i$ as their midpoint,
  and $n/2$ counts the others. Thus the total number of $3$-APs in $[n]$ is
  \[
    \frac13\sum_{i=1}^n f(i) = \frac{n^2}{4} + O(n),
  \]
  and therefore (using $np^3=o(1)$)
  \[
    \kappa_1 = \frac{n^2p^3}4 + o(1).
  \]
  If $\{i,j\}$ is an edge in the dependency graph, then $|\gamma_i\cap\gamma_j|$ is either $1$ or $2$.
  The number of pairs $\gamma_i,\gamma_j$ intersecting in two elements is at most $\binom{n}{2}\binom{3}{2}^2$,
  so the contribution of these pairs to $\kappa_2$ is $O(n^2p^4)$, which is $o(1)$ by our assumption on $p$.
  The number of pairs $\{\gamma_i,\gamma_j\}$ with $i \neq j$ and $|\gamma_i\cap\gamma_j| \geq 1$ is precisely
  $\sum_{i=1}^n \binom{f(i)}{2}$ and hence the number $M$ of pairs with $|\gamma_i \cap \gamma_j| = 1$ satisfies
  \[
    M = \sum_{i=1}^n \binom{f(i)}{2} + O(n^2) = \frac12\sum_{i=1}^n f(i)^2 + O(n^2).
  \]
  Since
  \[
    \begin{split}
      \sum_{i=1}^n f(i)^2 & = \sum_{i=1}^n \big(n/2 + \min{\{i,n-i\}}\big)^2 + O(n^2) = 2 \sum_{i=1}^{\lfloor n/2 \rfloor} (n/2+i)^2 + O(n^2) \\
      & = 2\left( \frac{n^3}{3} - \frac{(n/2)^3}{3} \right) + O(n^2) = \frac{7n^3}{12} + O(n^2)
    \end{split} 
  \]
  and $n^2p^4 = o(1)$, we have
  \[ 
    \kappa_2 = M(p^5-p^6) + O\big(n^2(p^4-p^6)\big) = \frac{7n^3p^5}{24} + o(1).
  \]
  Lastly, we claim that $\Delta_3 = O(n^4p^7) = o(1)$.
  Since any two distinct numbers are contained in at most three $3$-APs,
  we have $|\mathcal C_3|= O(n^4)$.
  Moreover, let $\mathcal
  C_3^*$ be the family of all $\{i, j, k\} \in \mathcal C_3$ such that
  $|\gamma_i \cup \gamma_j \cup \gamma_k| < 7$. A simple case analysis shows that
  \[
    \sum_{V\in \mathcal C_3^*} \Delta(\{X_i: i\in V\}) = O(n^2p^5 + n^3p^6) = o(1).
  \]
  On the other hand, $\Delta(\{X_i: i\in V\}) = p^7$ for every $V \in \mathcal C_3 \setminus \mathcal C_3^*$. Thus,
  \[
    \Delta_3 \leq |\mathcal C_3| p^7 + \sum_{V\in \mathcal C_3^*} \Delta(\{X_i: i\in V\})
     = O(n^4p^7 + n^2p^4 + n^3p^6) = o(1)
  \]
  and we conclude that the probability that $[n]_p$ is $3$-AP-free is asymptotically
  \[
    \exp\Big(- \frac{n^2p^3}{4} + \frac{7n^3p^5}{24}\Big),
  \]
  as claimed.
\end{proof}

\paragraph{Acknowledgement}
This project was started during the workshop of the research group of Angelika
Steger in Buchboden in February 2014. We are grateful to the anonymous referee
for their careful reading of this paper and many helpful suggestions; in
particular, for pointing out a mistake in an earlier version of the paper.

\bibliographystyle{plain}
\bibliography{janson}

\end{document}